\documentclass[pdflatex,sn-mathphys-num]{sn-jnl}

\usepackage{graphicx}%
\usepackage{multirow}%
\usepackage{amsmath,amssymb,amsfonts}%
\usepackage{amsthm}%
\usepackage{mathrsfs}%
\usepackage[title]{appendix}%
\usepackage{xcolor}%
\usepackage{textcomp}%
\usepackage{manyfoot}%
\usepackage{booktabs}%
\usepackage{algorithm}%
\usepackage{algorithmicx}%
\usepackage{algpseudocode}%
\usepackage{listings}%

\theoremstyle{thmstyleone}%
\newtheorem{theorem}{Theorem}[section]
\newtheorem{proposition}[theorem]{Proposition}%
\newtheorem{lemma}{Lemma}[section]%

\theoremstyle{thmstyletwo}%

\theoremstyle{thmstylethree}%
\newtheorem{definition}{Definition}[section]%

\numberwithin{equation}{section}

\begin{document}

\title[Article Title]{Transposed $\delta$-Poisson (super)algebra Structures on the Virasoro-like algebra and its  Kantor Lie-double}


\author*[1]{\fnm{Jie} \sur{Lin}}\email{linj022@126.com}

\author[2]{\fnm{Chengyu} \sur{Liu}}\email{2563253950@qq.com}

\author[2]{\fnm{Jingjing} \sur{Jiang}}\email{745876258@qq.com}

\affil*[1]{\orgdiv{Sino-European Institute of Aviation Engineering }, \orgname{Civil Aviation University of China},  \city{Tianjin}, \postcode{300300},  \country{China}}

\affil[2]{\orgdiv{College of Science}, \orgname{Civil Aviation University of China}, \city{Tianjin}, \postcode{300300},  \country{China}}



\abstract{We undertake a study of transposed $\delta$-Poisson (super)algebra structures on the Virasoro-like algebra and its Kantor Lie-double —— the latter being constructed via Kantor’s procedure. This work leads to the finding that, whereas non-trivial $\delta$-derivations exist solely at $\delta=1$, non-trivial transposed $\delta$-Poisson (super)algebra structures are entirely absent.}




\maketitle

\section{Introduction}\label{sec1}


Transposed Poisson algebras was introduced in \cite{bib1} as a dual counterpart to the traditional Poisson algebra. This duality was established through the interchange of the roles of the two binary operations within the Leibniz rule that defines the Poisson algebra. In recent years, much work has been done on the relationship between transposed Poisson algebras and other algebras. It has proven that every transposed Poisson algebra is an $\mathbb F$-manifold, as established in \cite{bib2}. Recently, Fernández Ouaridi demonstrated in \cite{bib3} that a transposed Poisson algebra is simple if and only if its associated Lie bracket is simple. The authors in \cite{bib28} described $\frac 12$-derivations of solvable Lie algebras with a filiform nilradical and constructed nontrivial transposed Poisson algebras with solvable Lie algebras. 
Additionally, Beites, Ferreira, and Kaygorodov delved into the connections between transposed Poisson algebras and various other algebraic structures, such as Hom-Lie structures, quasi-automorphisms, and Poisson $n$-Lie algebras, as discussed in \cite{bib4}. They also proposed a set of open questions regarding transposed Poisson algebras for future exploration. 

Furthermore, Ferreira, Kaygorodov, and Lopatkin established a link between $\frac{1}{2}$-derivations of Lie algebras and transposed Poisson algebras in their work \cite{bib5}. This connection provides  a systematic approach to identifying all transposed Poisson structures associated with a given Lie algebra. Using this approach, they examined transposed Poisson structures within a range of Lie algebras, including the Witt algebra, Virasoro algebra, thin Lie algebra, solvable Lie algebra with an abelian nilpotent radical, and several others. This concept has been subsequently extended to characterize transposed Poisson structures on a spectrum of algebras and superalgebras, as detailed in references \cite{bib6}, \cite{bib7}, \cite{bib8} and \cite{bib9}, and  beyond. Recently, the applications and advancements in the study of $\frac{1}{2}$-derivations and transposed Poisson algebra structures  associated with specific Lie algebras have been discussed in \cite{bib10}. The text discusses the application of $\frac{1}{2}$-derivations in specific algebraic structures, as well as how to generalize these concepts to more general algebraic structures, including $n$-ary algebras and Hom-algebras.

In the paper\cite{bib13}, two more general algebras: $\delta$-Poisson and transposed $\delta$-Poisson algebras were studied. They are closely related to $\delta$-derivations introduced by Filippov in \cite{bib22}. It was shown in \cite{bib13} that transposed $\delta$-Poisson algebras share many known
similarities with those studied in \cite{bib1}. In \cite{bib14}, the authors classify all transposed $\delta$-Poisson structures on null-filiform associative algebras and provide a complete classification of transposed $\delta$-Poisson algebras corresponding to each value of the parameter $\delta$. 

The Virasoro algebra plays an important role in many braches in both mathematics and physics, it can be regarded as the universal central extension of the Lie algebra of derivations on the Laurent polynomial ring in one variable. In the paper \cite{bib21}, it was proved that the Virasoro algebra does not admit nontrivial transposed $\delta$-Poisson structures, and the nontrivial  transposed
$\delta$-Poisson structures on the Lie algebras $W(a, b)$ were described.

As an analogue of the Virasoro algebra, the authors in the paper \cite{bib23} introduced the Virasoro-like algebra, which is the skew derivation Lie algebra over the Laurent polynomial ring in two commuting variables. It is also a kind of $W_{\infty}$ algebras in Physics. The structure and representation theory on the Virasoro-like algebra have been studied by many scholars, for instance, see \cite{bib19}, \cite{bib24}, \cite{bib25}, \cite{bib26}, \cite{bib27} and the references therein.

Let us recall the details for the Virasoro-like algebra.
The Virasoro-like algebra $V$ can be regarded as a vector field over the ring of Laurent polynomials $\mathbb C[x_1^{\pm}, x_2^{\pm}]$ with complex coefficients. Let $L_{\bf m}=x_1^{m_1}x_2^{m_2}(m_2x_1\frac{\partial}{\partial x_1}-m_1x_2\frac{\partial}{\partial x_2}),$ where ${\bf m}=(m_1, m_2)\in\mathbb Z^2\setminus\{{\bf 0}\},$ then $V=\left<L_{\bf m}|{\bf m}\in\mathbb Z^2\setminus\{{\bf 0}\}\right>$ forms a Lie algebra with respect to the following defined operation $[\cdot,\cdot]$:
 \begin{equation}\label{eq1.01}
     [L_{\bf m}, L_{\bf n}]=\det\binom{\textbf{n}}{\textbf{m}}L_{\bf m+n},
     \end{equation}
where $\det\binom{\textbf{n}}{\textbf{m}}=n_1m_2-n_2m_1.$
If we set $L_{\bf 0}=0$ by convention and define $L_{\bf m}\cdot L_{\bf n}=L_{\bf m+n}$ for all ${\bf m, n}\in \mathbb Z^2\setminus \{{\bf 0}\}$, then $(V, \cdot, [\cdot, \cdot])$ forms a Poisson algebra. According to Kantor's definition (see \cite{bib12}), $(V, \cdot, [\cdot, \cdot])$ is also a generalized Poisson bracket. Then we know from \cite{bib12} that the algebra $L(V)=V\oplus V^s$, where $V^s=\left<G_{\bf m}\mid {\bf m}\in\mathbb Z^2\setminus\{{\bf 0}\}\right>,$ is a Lie superalgebra with respect to the following products, we call it the Kantor Lie-double of the Virasoro-like algebra $V$: 
\begin{eqnarray}
    [L_{\bf m}, L_{\bf n}]&=&\det\binom{\textbf{n}}{\textbf{m}}L_{\bf m+n}\label{eq1.1}\\
   \left[L_{\bf m}, G_{\bf n}\right]&=&\det\binom{\textbf{n}}{\textbf{m}}G_{\bf m+n}=[G_{\bf m}, L_{\bf n}]\label{eq1.2}\\ 
    \left[G_{\bf m}, G_{\bf n}\right]&=&L_{\bf m+n}\label{eq1.3}
\end{eqnarray}

As the supersymmetric extension of Virasoro-like algebras, the super-Virasoro-like algebras serve as pivotal mathematical instruments in the formulation of supersymmetric field theories. They offer an algebraic framework essential for characterizing physical systems that exhibit supersymmetry. These algebras act as a conduit, linking abstract mathematical constructs like Lie superalgebras and vertex algebras with concrete physical theories such as superstring theory and conformal field theory. In this manner, they not only enrich the theoretical framework of Lie superalgebras but also catalyze progress in the representation theory of infinite-dimensional Lie superalgebras. 

The paper is organized as follows. Section 2 is dedicated to the main definitions and Lemmas used in the paper. Section3 describes the transposed $\delta$-Poisson structures on the Virasoro-like algebras. Section 4 describes the transposed $\delta$-Poisson structures on the Kantor Lie-double of the Virasoro-like algebra.

\section{Preliminaries}\label{sec2}
\begin{definition}\textsuperscript{\cite{bib13}}\label{df2.1}
    Let $\delta$ be a fixed complex number. An algebra $(L, \cdot, [\cdot, \cdot])$ is defined to be transposed $\delta$-Poisson algebra, if $(L, \cdot)$ is a commutative associative algebra, $(L, [\cdot, \cdot])$ is a Lie algebra and the following identity holds: \begin{equation}\label{eq4.1}
        x\cdot[y, z]=\delta([x\cdot y, z]+[y, x\cdot z]), ~\forall ~x, y, z\in L.
    \end{equation}
\end{definition}

\begin{definition}\label{df4.1}
    Let $\delta$ be a fixed complex number, $L=L_{\overline 0}\oplus L_{\overline 1}$ be a $\mathbb Z_2$-graded vector space equipped with two bilinear operations $\cdot, [\cdot, \cdot]: L\otimes L\rightarrow L$, where $(L, \cdot)$ is a supercommutative associative superalgebra and $(L, [\cdot, \cdot])$ is a Lie superalgebra. The triple $(L, \cdot, [\cdot, \cdot])$ is called a transposed $\delta$-Poisson superalgebra if the following identity holds:
    \begin{equation}\label{eq2.2}
      x\cdot[y,  z]=\delta([x\cdot y, z]+(-1)^{|x||y|}[y, x\cdot z]), ~\forall ~x, y, z\in L_{\overline 0}\cup L_{\overline 1},   
    \end{equation}
     where $\left | x \right |\in \mathbb{Z}_2 $ denotes the degree of a homogeneous element $x\in L$.
\end{definition}

If we take $\delta=\frac 12$ in identity \eqref{eq4.1} and \eqref{eq2.2}, then we obtain the definitions of the transposed Poisson algebra and transposed Poisson superalgebra. A transposed $\delta$-Poisson (super)algebra $L$ is called trivial, if $L \cdot L= 0$ or $[L, L] = 0$.
\begin{definition}\label{df2.3}
    A homogeneous linear map $D: L \rightarrow L$ is called a $\delta$-derivation of a Lie algebra $(L, [\cdot, \cdot])$ if 
    \begin{equation}\label{eq2.3}
        D([x, y])=\delta ([D(x), y]+[x, D(y)]), ~\forall~ x, y\in L.
    \end{equation}
\end{definition}
\begin{definition}\label{df2.4}
    A homogeneous linear map $D: L \rightarrow L$ is called a $\delta$-superderivation of a Lie superalgebra $(L, [\cdot, \cdot])$ if 
    \begin{equation}\label{eq2.3}
        D([x, y])=\delta ([D(x), y]+(-1)^{|D||x|}[x, D(y)]), ~\forall~ x, y\in L_{\overline 0}\cup L_{\overline 1}.
    \end{equation}
\end{definition}

The basic example of a $\frac 12$-(super)derivation is the multiplication by a field element. Such $\frac 12$-(super)derivations will be called trivial. 

For a given Lie (super)algebra $L$, the space of all $\delta$-(super)derivations of $L$ will be denoted by ${\Delta}(L)$; the space of all (super)derivations of $L$ will be denoted by ${\rm Der}(L)$.  For any element $x$ of $L$, we can define the map ${\rm ad}(x): L\rightarrow L$ by ${\rm ad}(x)=[x,y].$ The linear maps ${\rm ad}(x)$ forms an ideal of ${\rm Der}(L)$ called the inner (super)derivation algebra of $L$, denotedy by ${\rm ad}(L)$.

 The following is a similar result to that from Farnsteiner \cite{bib15}. 
\begin{lemma}\label{C} Let $G$ be an abelian group and $L=\underset{g \in G}{\oplus} L_{g}$ be a $G$-graded Lie algebra. Assuming that $L$ is finitely generated as a Lie algebra, it follows that
	$$\Delta(L)=\underset{g\in G}{\oplus} \Delta_{g}(L).$$
\end{lemma}
Similarly, for a Lie superalgebra $L=L_{\bar{0}}\oplus L_{\bar{1}}$,
$$\Delta(L)=\Delta_{\bar{0}}(L)\oplus \Delta_{\bar{1}}(L).$$
where $\Delta_{\bar{0}}(L)$ and $\Delta_{\bar{1}}(L)$ denote the subspaces that comprise, respectively, the even and odd $\delta$-superderivations of the Lie (super)algebra $L$. Thus, for each $\delta$-superderivation $\varphi\in \Delta(L)$, we can deduce that $\varphi =\varphi _{\bar{0}}+\varphi _{\bar{1}}$, where $\varphi _{\bar{0}}\in \Delta _{\bar{0}}\left(L\right)$ and $\varphi _{\bar{1}}\in \Delta _{\bar{1}}\left(L\right)$.\\

Definitions \ref{df2.1}(resp. \ref{df4.1}) and \ref{df2.3}(resp. \ref{df2.4}) immediately imply the following Lemma.
\begin{lemma}\label{lem2.1}
    Let $(L, \cdot, [\cdot, \cdot])$ be a transposed $\delta$-Poisson (super)algebra and $z\in L$.  Then the left multiplication $L_{z}$ of $(L, \cdot)$ is a $\delta$-(super)derivation of $(L, [\cdot, \cdot])$ and $|L_{z}|=|z|$.
\end{lemma}
Theorem 8 in \cite{bib5} can be reformulated as follows with the same proof.
\begin{lemma}\label{lem2.3}
    Let $L$ be a Lie (super)algebra with $\dim L >1$. If $L$ has no non-trivial $\delta$-(super)derivations, then every transposed $\delta$-Poisson (super)algebra structure defined on $L$ is trivial.
\end{lemma}

The above lemmas  illustrate the connection between $\delta$-(super)derivations on a Lie (super)algebra and
transposed $\delta$-Poisson (super)algebra structures that can be defined on it. Utilizing these lemmas, one can ascertain the  presence of nontrivial transposed $\delta$-Poisson structures within a Lie (super)algebra by examining its $\delta$-(super)derivations.

\begin{definition}\textsuperscript{\cite{bib16}}
A Lie (super)algebra $(L, [\cdot, \cdot]$) is said to be prime if $[A, B]\neq \{0\}$ for any nonzero ideals $A, B\subseteq L.$
\end{definition}
It is easy to see that a simple Lie (super)algebra is a prime Lie (super)algebra.

Filippov studied $\delta$-derivations of prime Lie algebras and gave the following result:
\begin{lemma}\textsuperscript{\cite{bib18}}\label{lem2.2}
 A prime Lie algebra does not have nonzero $\delta$-derivations if $\delta\neq -1, 0, \frac 12, 1$.
 \end{lemma}
 Zusmanovich generalized  this result to the super case:
\begin{lemma}\textsuperscript{\cite{bib15}}\label{lem2.5}
 A prime Lie superalgebra does not have nonzero $\delta$-superderivations
if $\delta\neq -1, 0, \frac 12, 1$. 
\end{lemma}

For the transposed $0$-Poisson (super)algebra structures on a simple Lie (super)algebra, we have the following result:
\begin{theorem}\label{thm2.1}
    If $(L, [\cdot, \cdot])$ is a simple Lie (super)algebra, then there are no non-trivial transposed $0$-Poisson (super)algebra structures on $L$. 
\end{theorem}
\begin{proof} If $\delta=0,$ then identity \eqref{eq4.1} and \eqref{eq2.2} become $z\cdot[x,y]=0, ~\forall~ x, y, z\in L.$
 
 We assume $(L, \cdot, [\cdot, \cdot])$ is a transposed $0$-Poisson (super)algebra. Since $L$ is simple, for all $a, b\in L,$ there exist $x, y\in L$ such that $b=[x, y]$.  According to identity \eqref{eq4.1}, we obtain $ a\cdot b=0.$ 

 This shows that all transposed $0$-Poisson (super)algebra structures on the Lie (super)algebra $(L, [\cdot, \cdot])$ are trivial.
\end{proof}




Throughout this paper, we denote by $\mathbb{C}$, $\mathbb{Z}$, $\mathbb{N}$ the sets of all complex numbers, all integers and all positive integers, respectively,
and $\mathbb{C}^{*}=\mathbb{C}\setminus{\left \{ 0\right \} }$, $\mathbb{Z}^{*}=\mathbb{Z}\setminus{\left \{ 0\right \} }$, $\mathbb{N}^{*}=\mathbb{N}\setminus{\left \{ 0\right \} }$. We note that $\textbf{e}_{1}=\left(1,0\right)$ and $\textbf{e}_{2}=\left(0,1\right)\in \mathbb{Z} ^{2}$. Consequently, $\mathbb{Z}^{2}=\mathbb{Z}\textbf{e}_{1}\oplus \mathbb{Z}\textbf{e}_{2}$. Unless otherwise specified, we use $\textbf{m} = \left(m_{1}, m_{2}\right)$ to denote an element in $\mathbb{Z}^{2}$.
\section{Transposed $\delta$-Poisson algebra structure on Virasoro-like algebra}
In this section, we delve into an exploration of the transposed $\delta$-Poisson structures on the  Virasoro-like algebra. 
It is easy to see that there is a natural $\mathbb Z^2$-grading on $V$: $V=\underset{{\bf i}\in \mathbb Z^2}\oplus V_{\bf i},$ where $V_{\bf i}=\left<L_{\bf i} \right>$ for ${\bf i}\in \mathbb Z^2\setminus\{{\bf 0}\}$ and $V_{\bf 0}={\bf 0}.$ In addition, $V$ can be generated by the finite set $\{L_{(1,0)}, L_{(-1, 0)}, L_{(0, 1)}, L_{(0, -1)}\}.$

Ivan Kaygorodov and Mykola Khrypchenko investigated the transposed Poisson algebra structures on the Virasoro-like algebra in \cite{bib7} and drew the following conclusion.
\begin{lemma}\textsuperscript{\cite{bib7}}\label{lem3.1}
    There are no non-trivial transposed Poisson algebra structures on the Virasoro-like algebra $V$.
\end{lemma}
It is well known  that  the Virasoro-like algebra is simple, according to Lemma \ref{lem2.2}, Theorem \ref{thm2.1} and Lemma \ref{lem3.1}, it suffices to consider the cases $\delta=-1$ and $\delta=1$.
\subsection{Transposed $-1$-Poisson algebra structure on Virasoro-like algebra}
For $\delta=-1,$ a $\delta$-derivation of $V$ is a linear map $D: V\rightarrow V$ satisfies: $$D([x, y])=-[D(x), y]-[x, D(y)], ~\forall x, y\in V.$$
Firstly, we determine all the $-1$-derivations on $V$.
\begin{theorem}\label{thm3.02}
  The Virasoro-like algebra $V$ dose not have nonzero $-1$-derivations.  
\end{theorem}
\begin{proof}
    Let $D$ be a $-1$-derivation of $V$. According to Lemma \ref{C}, we can write $D=\sum\limits_{{\bf i}\in\mathbb Z^2}D_{\bf i}$, where $D_{\bf i}$ is also a $-1$-derivation of $V$. We assume $D_{\bf i}(L_{\bf m})=a_{\bf m}L_{\bf m+i}$ for all ${\bf m}\in \mathbb Z^2\setminus\{{\bf 0}\}.$ By the definition of $-1$-derivation and identity \eqref{eq1.01}, we have 
    \begin{equation}\label{eq3.02}
      \det\binom{\bf n}{\bf m} a_{\bf m+n}=-\det\binom{\bf n}{\bf m+i}a_{\bf m}-\det\binom{\bf n+i}{\bf m} a_{\bf n}, ~\forall~{\bf m, n}\in \mathbb Z^2\setminus\{{\bf 0}\}. 
    \end{equation}
 In order to determine the coefficients $a_{\bf m},$ we need to consider the following cases:\\
 {\bf Case 1.} ${\bf i=0}.$

 Taking ${\bf n=e_1}$ in identity \eqref{eq3.02}, we obtain
 \begin{equation}\label{eq3.003}
     m_2(a_{\bf m+e_1}+a_{\bf m}+a_{\bf e_1})=0, ~\forall~{\bf m}\in\mathbb Z^2\setminus\{\bf 0\}.
 \end{equation}
Then for all ${\bf m}\in \mathbb Z\times \mathbb Z^*, $ 
\begin{equation*}
    a_{\bf m+e_1}+\frac 12 a_{\bf e_1}=-(a_{\bf m}+\frac 12 a_{\bf e_1}).
\end{equation*}
   Fix $m_2\in\mathbb Z^*,$ and treat $(a_{\bf m}+\frac 12 a_{\bf e_1})_{m_1\in\mathbb Z}$ as a geometric sequence. Then the above identity yields 
   \begin{equation}\label{eq3.03}
       a_{\bf m}=(-1)^{m_1-1}(a_{(1, m_2)}+\frac 12 a_{\bf e_1})-\frac 12 a_{\bf e_1}, ~\forall~{\bf m}\in\mathbb Z\times\mathbb Z^*.
   \end{equation}
   Similarly, set ${\bf n=e_2}$ in identity \eqref{eq3.02}. Then we get 
 \begin{equation}\label{eq3.005}
     -m_1(a_{\bf m+e_2}+a_{\bf m}+a_{\bf e_2})=0, ~\forall~{\bf m}\in \mathbb Z^2\setminus\{{\bf 0}\}.
 \end{equation}
 By a similar argument we obtain 
   \begin{equation}\label{eq3.04}
      a_{\bf m}= (-1)^{m_2}(a_{(m_1, 0)}+\frac{1}2a_{\bf e_2})-\frac{1}2a_{\bf e_2}, ~\forall~{\bf m}\in\mathbb Z^*\times\mathbb Z.
   \end{equation}
   Taking $m_1=1$ in identity \eqref{eq3.04}, then we have
   \begin{equation*}
       a_{(1, m_2)}=(-1)^{m_2}a_{\bf e_1}+\frac 12((-1)^{m_2}-1)a_{\bf e_2}, ~\forall~m_2\in\mathbb Z.
   \end{equation*}
   Substituting the above identity into \eqref{eq3.03}, we get for all ${\bf m}\in\mathbb Z\times\mathbb Z^*$, 
   \begin{equation}\label{eq3.6}
        a_{\bf m}=((-1)^{m_1+m_2-1}+\frac 12((-1)^{m_1-1}-1))a_{\bf e_1}+\frac 12(-1)^{m_1-1}((-1)^{m_2}-1)a_{\bf e_2}.
   \end{equation}
   Substituting identity \eqref{eq3.6} into identity \eqref{eq3.02}, then for all ${\bf m, n}\in \mathbb Z\times\mathbb Z^*$ such that ${\bf m+n}\in \mathbb Z\times\mathbb Z^*,$  we have 
   \begin{eqnarray*}
       &&\det\binom{\bf n}{\bf m}(((-1)^{m_1+n_1+m_2+n_2-1}+\frac 12((-1)^{m_1+n_1-1}-1))a_{\bf e_1}+\frac 12(-1)^{m_1+n_1-1}((-1)^{m_2+n_2}-1)a_{\bf e_2}\\
       &&+((-1)^{m_1+m_2-1}+\frac 12((-1)^{m_1-1}-1))a_{\bf e_1}+\frac 12(-1)^{m_1-1}((-1)^{m_2}-1)a_{\bf e_2}\\
       &&+((-1)^{n_1+n_2-1}+\frac 12((-1)^{n_1-1}-1))a_{\bf e_1}+\frac 12(-1)^{n_1-1}((-1)^{n_2}-1)a_{\bf e_2})=0
   \end{eqnarray*}
       Taking ${\bf m}=(1, 1), {\bf n}=(2, 1)$ and ${\bf m}=(2, 2), {\bf n}=(2, 1)$ in this identity respectively, then we obtain $a_{\bf e_2}=0$ and $a_{\bf e_1}=0.$ Substituting these results into identity \eqref{eq3.6}, we get 
       $$a_{\bf m}=0, ~\forall ~{\bf m}\in\mathbb Z\times \mathbb Z^*.$$
       Particularly, $a_{(m_1, -1)}=0, ~\forall~ m_1\in\mathbb Z.$ Then taking $m_2=-1$ in identity \eqref{eq3.005}, we deduce $$a_{(m_1, 0)}=0, ~\forall ~m_1\in\mathbb Z^*.$$
       Therefore, for all ${\bf m}\in \mathbb Z^2\setminus\{\bf 0\}$,  $a_{\bf m}=0.$\\ 
           {\bf Case 2.} ${\bf i}\in\mathbb Z^*\times\{0\}$ or $\{0\}\times\mathbb Z^*.$  

   Without loss of generality, we suppose ${\bf i}\in \{0\}\times \mathbb Z^*.$

Taking ${\bf n=e_1}$ in identity \eqref{eq3.02}, we obtain
   \begin{equation*}
       m_2a_{\bf m+e_1}=-(m_2+i_2)a_{\bf m}-(m_2-m_1i_2)a_{\bf e_1}, ~\forall ~{\bf m}\in\mathbb Z\times\mathbb Z^*.
   \end{equation*}
   Set $m_2=0$, then we have 
   \begin{equation}\label{eq3.05}
       a_{(m_1, 0)}=m_1a_{\bf e_1}, ~\forall~ m_1\in\mathbb Z^*.
   \end{equation}
   Taking ${\bf n=e_2}$ in identity \eqref{eq3.02}, we get
   $a_{\bf m+e_2}=-(a_{\bf m}+(i_2+1)a_{\bf e_2}), ~\forall~{\bf m}\in\mathbb Z^*\times\mathbb Z.$ Fixing $m_1\in\mathbb Z^*$ and treating $(a_{(m_1, m_2)}+\frac{i_2+1}2a_{\bf e_2})_{m_2\in\mathbb Z}$ as a geometric sequence, then we have
\begin{equation*}
      a_{\bf m}= (-1)^{m_2}(a_{(m_1, 0)}+\frac{i_2+1}2a_{\bf e_2})-\frac{i_2+1}2a_{\bf e_2}, ~\forall~{\bf m}\in\mathbb Z^*\times\mathbb Z.
   \end{equation*}
   Substituting identity \eqref{eq3.05}, we obtain 
   \begin{equation}\label{eq3.09}
       a_{\bf m}= (-1)^{m_2}m_1a_{\bf e_1}+\frac{i_2+1}2((-1)^{m_2}-1)a_{\bf e_2}, ~\forall~{\bf m}\in\mathbb Z^*\times\mathbb Z.
   \end{equation}
    Substituting identity \eqref{eq3.09} into identity \eqref{eq3.02}, then for all ${\bf m, n}\in \mathbb Z^*\times\mathbb Z$ such that ${\bf m+n}\in \mathbb Z^*\times\mathbb Z,$  we have 
\begin{eqnarray*}
    &&\det\binom{\bf n}{\bf m}((-1)^{m_2+n_2}(m_1+n_1)a_{\bf e_1}+\frac{i_2+1}2((-1)^{m_2+n_2}-1)a_{\bf e_2})\\
    &=&-\det\binom{\bf n}{\bf m+i}((-1)^{m_2}m_1a_{\bf e_1}+\frac{i_2+1}2((-1)^{m_2}-1)a_{\bf e_2})\\
    &&-\det\binom{\bf n+i}{\bf m}((-1)^{n_2}n_1a_{\bf e_1}+\frac{i_2+1}2((-1)^{n_2}-1)a_{\bf e_2})
\end{eqnarray*}
    Taking ${\bf m}=(1, 1), {\bf n}=(2, 2)$ and ${\bf m}=(1, 1), {\bf n}=(3, 3)$ in identity \eqref{eq3.02} respectively, we have
   $$ \begin{cases}
    i_2(2a_{\bf e_1}+(i_2+1)a_{\bf e_2})=0 \\
   i_2(i_2+1)a_{\bf e_2}=0
\end{cases}
 $$
 \begin{itemize}
     \item If $i_2\neq -1,$ then we get immediately  $a_{\bf e_1}=a_{\bf e_2}=0.$ And from identity \eqref{eq3.09}, we obtain $a_{\bf m}=0, ~\forall~{\bf m}\in \mathbb Z^*\times\mathbb Z.$
     \item If $i_2=-1$, then we get $a_{\bf e_1}=0$ from the system of linear equations. And from identity \eqref{eq3.09}, we obtain  $a_{\bf m}=0, ~\forall~{\bf m}\in \mathbb Z^*\times\mathbb Z.$ 
 \end{itemize}
 By a similar argument as in Case 1, we can prove $a_{\bf m}=0, ~\forall~ {\bf m}\in\mathbb Z^2\setminus\{\bf 0\}$.\\
   {\bf Case 3.} ${\bf i}\in \mathbb Z^*\times\mathbb Z^*.$

   Taking ${\bf n=e_1}$ in identity \eqref{eq3.02}, we obtain
   \begin{equation*}
      m_2a_{\bf m+e_1}=-(m_2+i_2)a_{\bf m}-(m_2(1+i_1)-m_1i_2)a_{\bf e_1}. 
   \end{equation*}
   Setting $m_2=0$, then we get \begin{equation}\label{eq3.010}
       a_{(m_1, 0)}=m_1a_{\bf e_1}, ~\forall~ m_1\in\mathbb Z^*.
   \end{equation}
   Taking ${\bf n=e_2}$ in identity \eqref{eq3.02}, we obtain
   \begin{equation*}
       -m_1a_{\bf m+e_2}=(m_1+i_1)a_{\bf m}-(m_2i_1-m_1(i_2+1))a_{\bf e_2}.
   \end{equation*}
   Setting $m_1=0$, we get 
   \begin{equation}\label{eq3.011}
       a_{(0, m_2)}=m_2a_{\bf e_2}, ~\forall ~m_2\in\mathbb Z^*.
   \end{equation}
   Taking ${\bf m}=(m_1, 0), {\bf n}=(0, m_2)$ in identity \eqref{eq3.02} and substituting \eqref{eq3.010} and \eqref{eq3.011}, we deduce 
 \begin{equation}\label{eq3.012}
     a_{(m_1, m_2)}=-(m_1+i_1)a_{\bf e_1}-(m_2+i_2)a_{\bf e_2}, ~\forall~(m_1, m_2)\in \mathbb Z^*\times\mathbb Z^*.
 \end{equation}
 Substituting the above identity into identity \eqref{eq3.02}, the for all ${\bf m, n}\in \mathbb Z^*\times \mathbb Z^*$ such that ${\bf m+n}\in \mathbb Z^*\times \mathbb Z^*$, we have 
 \begin{eqnarray*}
     && \det\binom{\bf n}{\bf m}(-(m_1+n_1+i_1)a_{\bf e_1}-(m_2+n_2+i_2)a_{\bf e_2})\\
     &=&\det\binom{\bf n}{\bf m+i}((m_1+i_1)a_{\bf e_1}+(m_2+i_2)a_{\bf e_2})+\det\binom{\bf n+i}{\bf m}((n_1+i_1)a_{\bf e_1}+(n_2+i_2)a_{\bf e_2})
 \end{eqnarray*}
  Setting ${\bf m}=(-2i_1, i_2), {\bf n}=(-i_1, i_2)$ and ${\bf m}=(i_1, -i_2), {\bf n}=(-i_1, i_2)$ in the above identity respectively, we obtain
$$ \begin{cases}
    -i_1a_{\bf e_1}+9i_2a_{\bf e_2}=0 \\
   i_1a_{\bf e_1}+2i_2a_{\bf e_2}=0
\end{cases}
 $$    
 Then we get immediately that $a_{\bf e_1}=a_{\bf e_2}=0.$  Substituting this result into identities \eqref{eq3.010}-\eqref{eq3.012}, we have 
 $$a_{\bf m}=0, ~\forall~ {\bf m}\in\mathbb Z^2\setminus\{0\}.$$
 Combining all the three cases above, we deduce that $D=0.$
\end{proof}
\begin{theorem}
    There are no nontrivial transposed $-1$-Poisson algebra structures on the Virasoro-like algebra $(V, [\cdot, \cdot])$.
\end{theorem}
\subsection{Transposed $1$-Poisson algebra structure on Virasoro-like algebra}
For $\delta=1$, a $\delta$-derivation of $V$ is exactly a derivation of $V$. Jiang C. and Meng D. determined all the derivations of $V$ in \cite{bib19}. 
\begin{lemma}\label{lem3.2}
    ${\rm Der}(V)={\rm ad}(V)+\mathbb C d_1+\mathbb C d_2,$ where $d_1$ and $d_2$ are defined as follow:
$$      \forall ~ {\bf m}\in \mathbb Z^2\setminus\{{\bf 0}\}, ~d_1(L_{\bf m})=m_1L_{\bf m}, ~ d_2(L_{\bf m})=m_2L_{\bf m}.
$$
\end{lemma}
\begin{theorem}
  There are no non-trivial transposed $1$-Poisson algebra structures on the Virasoro-like algebra $(V, [\cdot, \cdot])$.
\end{theorem}
\begin{proof}
    Let $(V, \cdot, [\cdot, \cdot])$ be a transposed $1$-Poisson algebra.  According to Lemma \ref{lem2.2}, for all $x\in V$, $L_x$ is a derivation of $(V, [\cdot, \cdot])$. By Lemma \ref{lem3.2}, we know that for all
  ${\bf m, n}\in \mathbb Z^2\setminus\{{\bf 0}\},$ 
  \begin{eqnarray*}
       L_{\bf m}\cdot L_{\bf n}&=&\sum\limits_{{\bf i}\in \mathbb Z^2\setminus\{{\bf 0}\}}a_{\bf i}[L_{\bf i}, L_{\bf n}]+(c_{L_{\bf m}}n_1+d_{L_{\bf m}}n_2)L_{\bf n}\\
       &=&\sum\limits_{{\bf i}\in \mathbb Z^2\setminus\{{\bf 0}\}}a_{\bf i} \det\binom{\bf n}{\bf i}L_{\bf n+i}+(c_{L_{\bf m}}n_1+d_{L_{\bf m}}n_2)L_{\bf n}.
  \end{eqnarray*}
  \begin{eqnarray*}
  L_{\bf n}\cdot L_{\bf m}&=&\sum\limits_{{\bf i}\in \mathbb Z^2\setminus\{{\bf 0}\}}b_{\bf i}[L_{\bf i}, L_{\bf m}]+(c_{L_{\bf n}}m_1+d _{L_{\bf n}}m_2)L_{\bf m}\\
  &=&\sum\limits_{{\bf i}\in \mathbb Z^2\setminus\{{\bf 0}\}}b_{\bf i}\det\binom{\bf m}{\bf i}L_{\bf m+i}+(c_{L_{\bf n}}m_1+d _{L_{\bf n}}m_2)L_{\bf m}.
  \end{eqnarray*}
  Since $\cdot$ is commutative, then 
  $L_{\bf m}\cdot L_{\bf n}-L_{\bf n}\cdot L_{\bf m}=0.$ We know that for all ${\bf i}\in\mathbb Z^2\setminus\{{\bf 0}\},$ if ${\bf m\neq n},$ then ${\bf m+i}\neq {\bf n+i}.$ So for all ${\bf i}\in\mathbb Z^2\setminus\{{\bf 0}\},$ $a_{\bf i}\det\binom{\bf n}{\bf i}=b_{\bf i}\det\binom{\bf m}{\bf i}=0$ and $c_{L_{\bf m}}n_1+d_{L_{\bf m}}n_2=c_{L_{\bf n}}m_1+d _{L_{\bf n}}m_2=0$
 because of $(L_{\bf m})_{{\bf m}\in\mathbb Z^2\setminus\{{\bf 0}\}}$ is linear independent.
  Thus, $\forall~ {\bf m, n}\in \mathbb Z^2\setminus\{{\bf 0}\}, L_{\bf m}\cdot L_{\bf n}=0,$ and $(V, \cdot, [\cdot, \cdot])$ is a trivial transposed $1$-Poisson algebra.   
\end{proof}
Based on the discussion in this section, we arrive at the following conclusion.
\begin{theorem}
  The Virasoro-like algebra $V$ admits no nontrivial transposed $\delta$-Poisson algebra structures.  
\end{theorem}
\section{Transposed $\delta$-Poisson superalgebra structure on the Kantor Lie-double of Virasoro-like algebra}\label{sec3}

In this section, we delve into an exploration of the transposed $\delta$-Poisson structures on the Kantor Lie-double of the Virasoro-like algebra, which is denoted by $L(V)$, as introduced in Section \ref{sec1}.

\begin{proposition}
The Lie superalgebra $L(V)$ is simple.
\end{proposition}
\begin{proof}
Let $I$ be a nonzero ideal of $L(V)$, and $x=\sum\limits_{i=1}^ka_iL_{(m_i, n_i)}+\sum\limits_{j=1}^{\ell} b_jG_{(s_j, t_j)}$ be the element in $I\setminus \{0\}$ such that $k+\ell$ is minimal.

{\bf Case 1.} $\ell\geq 1.$ Then 
$$[x, G_{(-s_1, -t_1)}]=\sum\limits_{i=1}^{k} a_i(-s_1n_i+t_1m_i)G_{(m_i-s_1, n_i-t_1)}+\sum\limits_{j=2}^{\ell}b_jL_{(s_j-s_1, t_j-t_1)}\in I.$$
By the minimality of $k+\ell$ we obtain $\ell=1$ and for all $i\in \{1, 2, \cdots, k\}, -s_1n_i+t_1m_i=0.$ Thus
$$[x, G_{(s_1, t_1)}]=\sum\limits_{i=1}^{k} a_i(s_1n_i-t_1m_i)G_{(m_i+s_1, n_i+t_1)}+b_1L_{(2s_1, 2t_1)}\in I.$$
By the minimality of $k+\ell$ again, we know that $k=0$ and $x=b_1G_{(s_1, t_1)}.$

Therefore, there exists $(s,t)\in\mathbb Z^2\setminus\{{\bf 0}\}$ such that $G_{(s, t)}\in I$. Then, for all $y\in L(V)_{\overline 0},$  there exists $z\in L(V)_{\overline 1}$ such that  $y=[G_{(s, t)}, z].$ Since $G_{(s, t)}\in I$, then $y\in I.$ This shows that $L(V)_{\overline 0}\subseteq I.$ Now we consider the following two subcases for proving $L(V)_{\overline 1}\subseteq I.$

{\bf Subcase 1.} $s=0.$ Then $t\neq 0$, and for all $(m, n)\in\mathbb Z^*\times\mathbb Z,$ $$G_{(m,n)}=-\frac{1}{mt}[L_{(m, n-t)}, G_{(0, t)}]\in I.$$
In addition, for all $n\in\mathbb Z^*$, there exists $(m, p)\in \mathbb Z^*\times \mathbb Z$, such that $$G_{(0, n)}=\frac 1{mn}[L_{(-m, p)}, G_{(m, n-p)}].$$ Since $G_{(m, n-p)}\in I,$ then we know $G_{(0, n)}\in I.$

{\bf Subcase 2.} $s\neq 0.$ Then $G_{(0,1)}=\frac 1s[L_{(-s, 1-t)}, G_{(s, t)}]\in I.$  By {\bf Subcase 1}, we know $L(V)_{\overline 1}\subseteq I.$

Therefore, $I=L(V).$\\
{\bf Case 2.} $\ell=0.$ Then $x=\sum\limits_{i=1}^ka_iL_{(m_i, n_i)}.$ Since $L(V)_{\overline 0}$ is the Virasoro-like algebra, and it is simple, then it is easy to see $L(V)_{\overline 0}\subseteq I.$ In addition,
 $$\forall ~(m, n)\in\mathbb Z^*\times \mathbb Z,~ G_{(m, n)}=[L_{(0,1)}, G_{(m ,n-1)}]\in I,$$
 $$\forall~ (m, n)\in\mathbb Z\times \mathbb Z^*,~ G_{(m, n)}=[L_{(1, 0)}, G_{(m-1, n)}]\in I.$$ 
 Thus, for all $(m ,n)\in \mathbb Z^2\setminus\{{\bf 0}\}, G_{(m ,n)}\in I.$ 

 Therefore, $L(V)_{\overline 1}\subseteq I$ and $I=L(V).$ 
 
 All in all, $L(V)$ is simple.
\end{proof}

According to Lemma \ref{lem2.5} and Theorem \ref{thm2.1}, it suffices to consider the cases: $\delta=-1,$ $\frac 12, 1$.\\


 The following Lemma is trivial: 
\begin{lemma}
 The Lie superalgebra $L(V)$ can be generated by the finite set $\left \{ L _{(0,\pm1)},L _{(\pm 1,0)},G _{(0,\pm1)},G _{(\pm 1,0)} \right \}$. 
\end{lemma}

Now we know that $L(V)$ is a finitely generated infinite-dimensional complex Lie superalgebra with $L(V)_{\overline 0}=\left<L_{\bf m}\mid {\bf m}\in \mathbb Z^2\setminus\{\bf 0\}\right>$ and $L(V)_{\overline 1}=\left<G_{\bf m}\mid {\bf m}\in\mathbb Z^2\setminus\{0\}\right>$. In addition, there is a natural $\mathbb{Z}^{2}$-grading on $L(V)$: $L(V)=\underset{{\bf i}\in\mathbb Z^2}\oplus L(V)_{\bf i}$,  where $L(V)_{\bf i}=\left<L_{\bf i}, G_{\bf i}\right>$ for ${\bf i}\in \mathbb Z^2\setminus\{\bf 0\}$ and $L(V)_{\bf 0}=0.$ By Lemma \ref{C}, we get 
\begin{equation}\label{eq3.1}
\Delta(L(V))=\underset{\textbf{i}\in \mathbb{Z}^{2}}{\oplus} \Delta_{\textbf{i}}(L(V)).
 \end{equation}
 To furnish a comprehensive characterization of the transposed $\delta$-Poisson superalgebra structures on $L(V)$, it suffices to calculate the odd and even $\delta$-superderivations on $L(V)$. Subsequently, leveraging Lemma \ref{lem2.3}, we can ascertain the complete set of transposed $\delta$-Poisson superalgebra structures on $L(V)$.
\subsection{Transposed Poisson superalgebra structure on $L(V)$}

Firstly,  we determine the odd $\frac 12$-superderivations of $L(V)$, that are the linear maps $\varphi : L(V)\to L(V)$, which are characterized by their property of satisfying
$$\varphi\left( L(V)_{\bar{0}}\right)\subseteq L(V)_{\bar{1}},~~~\varphi\left( L(V)_{\bar{1}}\right)\subseteq L(V)_{\bar{0}}.$$
In this case, $\left | \varphi \right |=1$, and $\varphi$ is a $\frac{1}{2}$-superderivation of $L(V)$ if and only if
$$\varphi\left ( \left[x,y\right] \right ) =\frac{1}{2} \left(\left[\varphi\left ( x\right ),y\right]+\left[x,\varphi\left ( y\right )\right]\right),~\forall~x\in L(V)_{\bar{0}},$$
$$\varphi\left ( \left[x,y\right] \right ) =\frac{1}{2} \left(\left[\varphi\left ( x\right ),y\right]-\left[x,\varphi\left ( y\right )\right]\right),~\forall~x\in L(V)_{\bar{1}}.$$

\begin{theorem}\label{A}
$\Delta _{\overline 1} \left ( L(V) \right ) =\{0\}$.
\end{theorem}
\begin{proof} Let $\varphi$ be an odd $\frac 12$-superderivation of $L(V)$, then by identity \eqref{eq3.1}, we can write $\varphi=\sum\limits_{{\bf i}\in\mathbb Z^2}\varphi_{\bf i}$, where $\varphi_{\bf i}$ is also an odd $\frac 12$-superderivation of $L(V)$ for all ${\bf i}\in \mathbb Z^2.$ 
Let $\textbf{i}\in \mathbb{Z}^{2}$, due $\varphi_{\bf i}\left( L(V)_{\bar{0}}\right)\subseteq L(V)_{\bar{1}}$, $\varphi_{\bf i}\left( L(V)_{\bar{1}}\right)\subseteq L(V)_{\bar{0}}$, 
we suppose 
\begin{equation*}\varphi _{\textbf{i}}\left(L_{\textbf{m}}\right)=a_{\textbf{m}} G_{\textbf{m}+\textbf{i}},\end{equation*}
\begin{equation*}\varphi _{\textbf{i}}\left(G_{\textbf{m}}\right)=b_{\textbf{m}} L_{\textbf{m}+\textbf{i}},\end{equation*}
for all $\textbf{m}\in\mathbb{Z}^{2}\setminus\left\{\bf 0\right\}$. Applying $\varphi _{\textbf{i}}$ to identities \eqref{eq1.1}-\eqref{eq1.3}, we obtain $\forall~ {\bf m, n}\in \mathbb Z^2\setminus\{\bf 0\},$
\begin{equation}\label{eq:0.4}
2\mathrm{det}\binom{\textbf{n}}{\textbf{m}} a_{\textbf{m}+\textbf{n}}=
\mathrm{det}\binom{\textbf{n}}{\textbf{m}+\textbf{i}}a_{\textbf{m}}+\mathrm{det}\binom{\textbf{n}+\textbf{i}}{\textbf{m}} a_{\textbf{n}},
\end{equation}
\begin{equation}\label{eq:0.5}
2\mathrm{det}\binom{\textbf{n}}{\textbf{m}} b_{\textbf{m}+\textbf{n}}=
a_{\textbf{m}}+\mathrm{det}\binom{\textbf{n}+\textbf{i}}{\textbf{m}} b_{\textbf{n}},\end{equation}
\begin{equation}\label{eq:0.6}
2a_{\textbf{m}+\textbf{n}}=
\mathrm{det}\binom{\textbf{n}}{\textbf{m}+\textbf{i}}b_{\textbf{m}}-\mathrm{det}\binom{\textbf{n}+\textbf{i}}{\textbf{m}} b_{\textbf{n}}.\end{equation}
 To determine the coefficients, we need to consider the following cases:

\textbf{Case 1.} $\textbf{i}={\bf 0}$.

 We know that for all ${\bf m}\in \mathbb Z^2\setminus\{{\bf 0}\},$ there exists an $n\in\mathbb Z^2\setminus\{{\bf 0}\},$ such that $\det\binom{\bf n}{\bf m}=0.$
 Thus by \eqref{eq:0.5}, we obtain $\forall~ {\bf m}\in \mathbb Z^2\setminus\{0\},~ a_{\bf m}=0. $  

Then the identities \eqref{eq:0.5} and \eqref{eq:0.6} become
\begin{equation}\mathrm{det}\binom{\textbf{n}}{\textbf{m}} \left (2b_{\textbf{m}+\textbf{n}}-b_{\textbf{n}}\right)=0,\label{eq:0.7}\end{equation}
\begin{equation}\mathrm{det}\binom{\textbf{n}}{\textbf{m}}\left (b_{\textbf{m}}-b_{\textbf{n}}\right)=0.\label{eq:0.8}\end{equation}
By taking $\textbf{m}=\textbf{e}_{2} $, $\textbf{n}=\textbf{e}_{1}$ in identity \eqref{eq:0.7}, and taking $\textbf{m}=\textbf{e}_{2}$, $\textbf{n}=\textbf{e}_{1}$; $\textbf{m}=\left(1,1\right) $, $\textbf{n}=\textbf{e}_{1}$ 
in identity \eqref{eq:0.8}, respectively, we have
$$\left\{\begin{array}{rcl}
2b_{(1,1)}-b_{\textbf{e}_{1}}&=&0\\
b_{\textbf{e}_{2}}-b_{\textbf{e}_{1}}&=&0\\
b_{(1,1)}-b_{\textbf{e}_{1}}&=&0
\end{array}\right.$$
Solving the above system of linear equations yields $b_{(1,1)}=b_{\textbf{e}_{1}}=b_{\textbf{e}_{2}}=0.$



By taking $\textbf{n}=\textbf{e}_{1}$ in identity \eqref{eq:0.8}, we get
$$m_{2}\left (b_{\textbf{m}}-b_{\textbf{e}_{1}}\right)=0,~\forall~ \textbf{m}\in\mathbb{Z}^{2}\setminus\left\{\bf 0\right\}.$$
Then
$$b_{\textbf{m}}=b_{\textbf{e}_{1}}=0,~\forall~\textbf{m}\in\mathbb{Z}\times\mathbb{Z}^{*}. $$
By taking $\textbf{n}=\textbf{e}_{2}$ in identity \eqref{eq:0.8}, we get
$$m_{1}\left (b_{\bf m}-b_{\textbf{e}_{2}}\right)=0,~\forall~{\bf m}\in\mathbb{Z}^2\setminus\{{\bf 0}\}.$$
Then
$$b_{\bf m}=b_{\textbf{e}_{2}}=0,~\forall~{\bf m}\in\mathbb{Z}^{*}\times \mathbb Z.$$
Thus we proved
$$b_{\textbf{m}}=0,~\forall~\textbf{m}\in\mathbb{Z}^{2}\setminus\left\{\bf 0\right\}.$$

\textbf{Case 2.} $\textbf{i}=(i_{1},i_{2})\in \left\{0\right\} \times \mathbb{Z}^{*}$ or $\mathbb{Z}^{*}\times\left\{0\right\} $. 

Without loss of generality, we suppose $\textbf{i}\in \left\{0\right\} \times \mathbb{Z}^{*}$. 

By identity \eqref{eq:0.4}, we have $\forall~ \textbf{m},\textbf{n}\in\mathbb{Z}^{2}\setminus\left\{\bf 0\right\},$
\begin{equation}\label{eq:9.20}
2\mathrm{det}\binom{\textbf{n}}{\textbf{m}}a_{\textbf{m}+\textbf{n}}=\left(\left(m_{2}+i_{2}\right)n_{1}-m_{1}n_{2}\right)a_{\textbf{m}}+
\left(m_{2}n_{1}-m_{1}\left(n_{2}+i_{2}\right)\right)a_{\textbf{n}}.\end{equation}
Taking $\textbf{n}= \textbf{e}_{1}$ and $\textbf{e}_{2}$ in identity \eqref{eq:9.20}, we get
\begin{equation}
2m_{2}a_{\textbf{m}+\textbf{e}_{1}}=\left(m_{2}+i_{2}\right)a_{\textbf{m}}+\left(m_{2}-m_{1}i_{2}\right)a_{\textbf{e}_{1}},
~\forall~\textbf{m}\in\mathbb{Z}^{2}\setminus\left\{\bf 0\right\},\label{eq:9.21}\end{equation}
\begin{equation}2m_{1}a_{\textbf{m}+\textbf{e}_{2}}=m_{1}a_{\textbf{m}}+m_{1}\left(1+i_{2}\right)a_{\textbf{e}_{2}},~\forall~  \textbf{m}\in\mathbb{Z}^{2}\setminus\left\{\bf 0\right\}.\label{eq:9.22}\end{equation}
Taking $m_{2}= 0$ in identity \eqref{eq:9.21}, we obtain
$$i_{2}\left(a_{(m_{1},0)}-m_{1}a_{\textbf{e}_{1}}\right)=0,~\forall  ~m_{1}\in\mathbb{Z}^{*}.$$
Since $i_{2}\not=0$, then
\begin{equation}a_{(m_{1},0)}=m_{1}a_{\textbf{e}_{1}},~~\forall  ~m_{1}\in\mathbb{Z}^{*}.\label{eq:9.24}\end{equation}
By identity \eqref{eq:9.22}, we have
\begin{equation}2a_{(m_{1},m_{2}+1)}=a_{\textbf{m}}+\left(1+i_{2}\right)a_{\textbf{e}_{2}},~\forall~\textbf{m}\in\mathbb{Z}^{*}\times\mathbb{Z}.\label{eq:9.23}\end{equation}
Fix $m_1\in\mathbb Z^*,$ and treat $\left(a_{\textbf{m}}-\left(1+i_{2}\right)a_{\textbf{e}_{2}}\right)_{m_{2}\in\mathbb{Z}}$ as a geometric sequence, then by identity \eqref{eq:9.23}, we have
\begin{equation*}a_{\textbf{m}}=\left( a_{(m_{1},0)}-\left(1+i_{2}\right)a_{\textbf{e}_{2}}\right)\left(\frac{1}{2}\right)^{m_{2}}+\left(1+i_{2}\right)a_{\textbf{e}_{2}},
~\forall ~\textbf{m}\in\mathbb{Z}^{*}\times\mathbb{Z}.\label{eq:9.25} \end{equation*}
By substituting identity \eqref{eq:9.24} into the identity  above, we get
\begin{equation}a_{\textbf{m}}=m_{1}\left(\frac{1}{2}\right)^{m_{2}}a_{\textbf{e}_{1}}+\left(1+i_{2}\right)\left(1-\left(\frac{1}{2}\right)^{m_{2}}\right)a_{\textbf{e}_{2}},
~\forall~  \textbf{m}\in\mathbb{Z}^{*}\times\mathbb{Z}.
\label{eq:9.26}\end{equation}
By substituting identity \eqref{eq:9.26} into identity \eqref{eq:9.20}, for those $\textbf{m},\textbf{n}\in\mathbb{Z}^{*}\times\mathbb{Z}$ such that $\textbf{m}+\textbf{n}\in\mathbb{Z}^{*}\times\mathbb{Z}$, , we have
\begin{equation}\begin{split}2\mathrm{det}\binom{\textbf{n}}{\textbf{m}}\left(\left(m_{1}+n_{1}\right)\left(\frac{1}{2}\right)^{m_{2}+n_{2}} a_{\textbf{e}_{1}}+
\left(1+i_{2}\right)\left(1-\left(\frac{1}{2}\right)^{m_{2}+n_{2}}\right) a_{\textbf{e}_{2}}\right)\\=
\left(\left(m_{2}+i_{2}\right)n_{1}-m_{1}n_{2}\right)\left(m_{1} \left(\frac{1}{2}\right)^{m_{2}}a_{\textbf{e}_{1}}+\left(1+i_{2}\right)\left(1-\left(\frac{1}{2}\right)^{m_{2}}\right) a_{\textbf{e}_{2}}\right)
\\+\left(m_{2}n_{1}-m_{1}\left(n_{2}+i_{2}\right)\right)\left(n_{1}\left(\frac{1}{2}\right)^{n_{2}}a_{\textbf{e}_{1}}+\left(1+i_{2}\right)\left(1-\left(\frac{1}{2}\right)^{n_{2}}\right) a_{\textbf{e}_{2}}\right).\end{split}\label{eq:9.27}\end{equation}

Now we need consider the two subcases:：$i_{2}= -1$ and $i_{2}\ne -1$.

\textbf{Subcase (i).} $i_{2}= -1$. Taking $\textbf{m}=(2,2)$, $\textbf{n}=(1,1)$ in identity \eqref{eq:9.27}, we get
$a_{\textbf{e}_{1}}=0.$
By substituting $i_{2}= -1$ and $a_{\textbf{e}_{1}}=0$ into identity \eqref{eq:9.26}, we get
$$a_{\textbf{m}}=0,~\forall~\textbf{m}\in\mathbb{Z}^{*}\times\mathbb{Z}.$$


\textbf{Subcase (ii).} $i_{2}\not= -1$. Taking $\textbf{m}=(1,1)$, $\textbf{n}=(2,2)$ and $\textbf{m}=(1,1)$, $\textbf{n}=(-2,-2)$ in identity \eqref{eq:9.27}, we get
$$\left\{\begin{matrix}
2a_{\textbf{e}_{1}}+\left(1+i_{2}\right)a_{\textbf{e}_{2}}=0,
 \\7a_{\textbf{e}_{1}}+2\left(1+i_{2}\right)a_{\textbf{e}_{2}}=0.
\end{matrix}\right.$$
Solving the system of linear equations yields
\begin{equation*}a_{\textbf{e}_{1}}=a_{\textbf{e}_{2}}=0.\label{eq:9.93}\end{equation*}
Taking $a_{\textbf{e}_{1}}=0$ and $a_{\textbf{e}_{2}}=0$ in \eqref{eq:9.26}, we get
$$a_{\textbf{m}}=0,~\forall~\textbf{m}\in\mathbb{Z}^{*}\times\mathbb{Z}.$$
Thus, in both of the two subcases, we proved $a_{\textbf{m}}=0,~\forall~\textbf{m}\in\mathbb{Z}^{*}\times\mathbb{Z}.$

By substituting $m_{1} = 1$, $\textbf{n}=(-1,0)$ in identity \eqref{eq:9.20}, we get
 $$a_{(0, m_2)}=0, ~\forall~ m_2\in\mathbb Z^{*}.$$

In summary, $$a_{\textbf{m}}=0,~\forall~\textbf{m}\in\mathbb{Z}^{2}\setminus\left\{\bf 0\right\}.$$
Thus, identities \eqref{eq:0.5} and \eqref{eq:0.6} transform into the following forms:
\begin{equation*}\label{eq:0.13}
2\mathrm{det}\binom{\textbf{n}}{\textbf{m}}b_{\textbf{m}+\textbf{n}}=\left(m_{2}n_{1}-m_{1}\left(n_{2}+i_{2}\right)\right)b_{\textbf{n}},
~\forall~\textbf{m},\textbf{n}\in\mathbb{Z}^{2}\setminus\left\{\bf 0\right\}\end{equation*}
\begin{equation*}\label{eq:0.15}
\left(m_{2}n_{1}-m_{1}\left(n_{2}+i_{2}\right)\right)b_{\textbf{n}}=\left(\left(m_{2}+i_{2}\right)n_{1}-m_{1}n_{2}\right)b_{\textbf{m}},
~\forall~\textbf{m},\textbf{n}\in\mathbb{Z}^{2}\setminus\left\{\bf 0\right\}.\end{equation*}
Adding the two identities above yields $\forall~\textbf{m},\textbf{n}\in\mathbb{Z}^{2}\setminus\left\{\bf 0\right\},$
\begin{equation}\label{eq:0.16}
4\mathrm{det}\binom{\textbf{n}}{\textbf{m}}b_{\textbf{m}+\textbf{n}}=
\left(\left(m_{2}+i_{2}\right)n_{1}-m_{1}n_{2}\right)b_{\textbf{m}}+\left(m_{2}n_{1}-m_{1}\left(n_{2}+i_{2}\right)\right)b_{\textbf{n}},
 \end{equation}
Taking $\textbf{n}= \textbf{e}_{1}$ and $\textbf{e}_{2}$ in identity \eqref{eq:0.16}, respectively, we get
\begin{equation}\label{eq:0.17}
4m_{2}b_{\textbf{m}+\textbf{e}_{1}}=\left(m_{2}+i_{2}\right)b_{\textbf{m}}+\left(m_{2}-m_{1}i_{2}\right)b_{\textbf{e}_{1}},
~\forall ~\textbf{m}\in\mathbb{Z}\times\mathbb{Z}^{*},\end{equation}
\begin{equation}\label{eq:0.18}
4b_{\textbf{m}+\textbf{e}_{2}}=b_{\textbf{m}}+\left(1+i_{2}\right)b_{\textbf{e}_{2}},
~\forall ~\textbf{m}\in\mathbb{Z}^{*}\times\mathbb{Z}.\end{equation}
Taking $m_{2}= 0$ in identity \eqref{eq:0.17}, we get
$$i_{2}\left(b_{(m_{1},0)}-m_{1}b_{\textbf{e}_{1}}\right)=0,~\forall~m_{1}\in\mathbb{Z}^{*}.$$
Since $i_{2}\not=0$, we get
\begin{equation}b_{(m_{1},0)}=m_{1}b_{\textbf{e}_{1}},~\forall  ~m_{1}\in\mathbb{Z}^{*}.\label{eq:0.19}\end{equation}
 Fix $m_1\in\mathbb Z^*$ and treat $\left(b_{\textbf{m}}-\left(\frac{1+i_{2}}{3}\right)b_{\textbf{e}_{2}}\right)_{m_{2}\in\mathbb{Z}}$ as a geometric sequence,  then by identity \eqref{eq:0.18} we have
\begin{equation*}b_{\textbf{m}}=\left( b_{(m_{1},0)}-\left(\frac{1+i_{2}}{3}\right)b_{\textbf{e}_{2}}\right)\left(\frac{1}{4}\right)^{m_{2}}+\left(\frac{1+i_{2}}{3}\right)b_{\textbf{e}_{2}},
~\forall ~\textbf{m}\in\mathbb{Z}^{*}\times\mathbb{Z}.\label{eq:0.21} \end{equation*}
By substituting identity \eqref{eq:0.19} into the identity above, we obtain
\begin{equation}\label{eq:0.21}
b_{\textbf{m}}=m_{1}\left(\frac{1}{4}\right)^{m_{2}}b_{\textbf{e}_{1}}+\left(\frac{1+i_{2}}{3}\right)\left(1-\left(\frac{1}{4}\right)^{m_{2}}\right)b_{\textbf{e}_{2}},
~\forall ~ \textbf{m}\in\mathbb{Z}^{*}\times\mathbb{Z}.
\end{equation}
By substituting identity \eqref{eq:0.21} into identity \eqref{eq:0.16}, for those $\textbf{m},\textbf{n}\in\mathbb{Z}^{*}\times\mathbb{Z}$, such that $\textbf{m}+\textbf{n}\in\mathbb{Z}^{*}\times\mathbb{Z}$ ,  we get
\begin{equation}\begin{split}4\mathrm{det}\binom{\textbf{n}}{\textbf{m}}\left(\left(m_{1}+n_{1}\right)\left(\frac{1}{4}\right)^{m_{2}+n_{2}} b_{\textbf{e}_{1}}+
\left(\frac{1+i_{2}}{3}\right)\left(1-\left(\frac{1}{4}\right)^{m_{2}+n_{2}}\right) b_{\textbf{e}_{2}}\right)\\
=
\left(\left(m_{2}+i_{2}\right)n_{1}-m_{1}n_{2}\right)\left(m_{1} \left(\frac{1}{4}\right)^{m_{2}}b_{\textbf{e}_{1}}+\left(\frac{1+i_{2}}{3}\right)\left(1-\left(\frac{1}{4}\right)^{m_{2}}\right) b_{\textbf{e}_{2}}\right)
\\+\left(m_{2}n_{1}-m_{1}\left(n_{2}+i_{2}\right)\right)\left(n_{1}\left(\frac{1}{4}\right)^{n_{2}}b_{\textbf{e}_{1}}+\left(\frac{1+i_{2}}{3}\right)\left(1-\left(\frac{1}{4}\right)^{n_{2}}\right) b_{\textbf{e}_{2}}\right).\end{split}\label{eq:0.22}\end{equation}

Now we need to consider the two subcases: $i_{2}= -1$ and $i_{2}\ne -1$.

\textbf{Subcase (i).} $i_{2}= -1$. Taking $\textbf{m}=(2,2)$, $\textbf{n}=(2,1)$ in identity \eqref{eq:0.22}, we have
$b_{\textbf{e}_{1}}=0.$
By substituting $i_{2}= -1$ and $b_{\textbf{e}_{1}}=0$ into identity \eqref{eq:0.21}, we get
$$b_{\textbf{m}}=0,~\forall~\textbf{m}\in\mathbb{Z}^{*}\times\mathbb{Z}.$$

\textbf{Subcase (ii).} $i_{2}\not= -1$. Taking $\textbf{m}=(1,1)$, $\textbf{n}=(-2,-2)$ and $\textbf{m}=(1,1)$, $\textbf{n}=(2,2)$ in identity \eqref{eq:0.22}, respectively, we obtain
$$\left\{\begin{matrix}
7b_{\textbf{e}_{1}}+\left(1+i_{2}\right)b_{\textbf{e}_{2}}=0,
 \\2b_{\textbf{e}_{1}}+\left(1+i_{2}\right)b_{\textbf{e}_{2}}=0.
\end{matrix}\right.$$
Solving the system of linear equations yields
\begin{equation*}b_{\textbf{e}_{1}}=b_{\textbf{e}_{2}}=0.\end{equation*}
Substuting $b_{\textbf{e}_{1}}=0$ and $b_{\textbf{e}_{2}}=0$ into identity \eqref{eq:0.21}, we have
$$b_{\textbf{m}}=0,~\forall~\textbf{m}\in\mathbb{Z}^{*}\times\mathbb{Z}.$$
Thus, in both of the two subcases we proved
$$b_{\textbf{m}}=0,~\forall~\textbf{m}\in\mathbb{Z}^{*}\times\mathbb{Z}.$$

Taking $n_{1} = -1$, $\textbf{m}=(1,0)$ in identity \eqref{eq:0.16}, we get
$$b_{(0,n_{2})}=0,~\forall  ~n_{2}\in\mathbb{Z}^{*}.$$
In summary, we proved
$$b_{\textbf{m}}=0,~\forall~\textbf{m}\in\mathbb{Z}^{2}\setminus\left\{\bf 0\right\}.$$

\textbf{Case 3.} $\textbf{i}\in\mathbb{Z}^{*}\times\mathbb{Z}^{*}$. 

On the one hand, taking $\textbf{n}=(1,0)$, $m_{2}= 0$ in identity \eqref{eq:0.4}, we get
\begin{equation}a_{(m_{1},0)}=m_{1}a_{\textbf{e}_{1}}, ~\forall  ~m_{1}\in\mathbb{Z}^{*}.\label{eq:9.33}\end{equation}
On the other hand, taking $\textbf{n}=(0,1)$, $m_{1}= 0$ in \eqref{eq:0.4}, we have
\begin{equation}a_{(0,m_{2})}=m_{2}a_{\textbf{e}_{2}},~\forall  ~m_{2}\in\mathbb{Z}^{*}.\label{eq:9.34}\end{equation}
Taking $n_{1} = 0$ and $m_{2} = 0$ in identity \eqref{eq:0.4}, we have
\begin{equation*}2m_{1}n_{2}a_{(m_{1},n_{2})}=n_{2}\left(m_{1}+i_{1}\right)a_{(m_{1},0)}+m_{1} \left(n_{2}+i_{2} \right)a_{(0,n_{2})},~\forall  ~m_{1},n_{2}\in\mathbb{Z}^{*}.\end{equation*}
Substituting identities \eqref{eq:9.33} and\eqref{eq:9.34} into the above identity, we get
\begin{equation*}2m_{1}n_{2}a_{(m_{1},n_{2})}=n_{2}\left(m_{1}+i_{1}\right)m_{1}a_{\textbf{e}_{1}}+m_{1}\left(n+i_{2}\right)n_{2}a_{\textbf{e}_{2}},~\forall  ~m_{1},n_{2}\in\mathbb{Z}^{*}.\end{equation*}
Then  we have
\begin{equation}2a_{(m_{1},n_{2})}=\left(m_{1}+i_{1}\right)a_{\textbf{e}_{1}}+\left(n_{2}+i_{2}\right)a_{\textbf{e}_{2}},~\forall  ~m_{1},n_{2}\in\mathbb{Z}^{*}.\label{eq:9.36}\end{equation}
By substituting \eqref{eq:9.36} into \eqref{eq:0.4}, for those $\textbf{m},\textbf{n}\in\mathbb{Z}^{*}\times\mathbb{Z}^{*}$ such that ${\bf m+n}\in\mathbb Z^*\times\mathbb Z^*$, we get
\begin{eqnarray*}
&&2\mathrm{det}\binom{\textbf{n}}{\textbf{m}}\left(\left(m_{1}+n_{1}+i_{1}\right) a_{\textbf{e}_{1}}+\left(m_{2}+n_{2}+i_{2}\right)a_{\textbf{e}_{2}}\right)
\\
&=&\mathrm{det}\binom{\textbf{n}}{\textbf{m}+\textbf{i}}\left(\left(m_{1}+i_{1}\right)a_{\textbf{e}_{1}}+\left(m_{2}+i_{2}\right)a_{\textbf{e}_{2}}\right)+\mathrm{det}\binom{\textbf{n}+\textbf{i}}{\textbf{m}}\left(\left(n_{1}+i_{1}\right)a_{\textbf{e}_{1}}+\left(n_{2}+i_{2}\right)a_{\textbf{e}_{2}}\right).\end{eqnarray*}
Particularly,  taking $\textbf{m}=(-2i_{1},i_{2})$, $\textbf{n}=(-i_{1},i_{2})$ and $\textbf{m}=(i_{1},-i_{2})$, $\textbf{n}=(-i_{1},i_{2})$, respectively, we get
$$\left\{\begin{matrix}
5i_{1}a_{\textbf{e}_{1}}&=0, \\
i_{1}a_{\textbf{e}_{1}}+i_{2}a_{\textbf{e}_{2}}&=0.
\end{matrix}\right.$$
Hence, it immediately follows that
$$a_{\textbf{e}_{1}}=a_{\textbf{e}_{2}}=0.$$
Substituting $a_{\textbf{e}_{1}}=0$ and $a_{\textbf{e}_{2}}=0$ into identity \eqref{eq:9.36}, we get
\begin{equation*}a_{\textbf{m}}=0,~\forall~ \textbf{m}\in\mathbb{Z}^{*}\times \mathbb{Z}^{*}.\end{equation*}
Substituting $a_{\textbf{e}_{1}}=0$ into identity \eqref{eq:9.33}, we get
\begin{equation*}a_{(m_{1},0)}=0,~\forall~ m_{1}\in\mathbb{Z}^*.\end{equation*}
Substituting $a_{\textbf{e}_{2}}=0$ into identity \eqref{eq:9.34}, we get
\begin{equation*}a_{(0,m_{2})}=0,~\forall ~m_{2}\in\mathbb{Z}^*.\end{equation*}
In summary, we proved
$$a_{\textbf{m}}=0,~\forall~\textbf{m}\in\mathbb{Z}^{2}\setminus\left\{\bf 0\right\}.$$
Then identities \eqref{eq:0.5} and \eqref{eq:0.6} transform into the following forms:
\begin{equation*}\label{eq:0.24}
2\mathrm{det}\binom{\textbf{n}}{\textbf{m}}b_{\textbf{m}+\textbf{n}}=\mathrm{det}\binom{\textbf{n}+\textbf{i}}{\textbf{m}}b_{\textbf{n}},
~\forall~\textbf{m},\textbf{n}\in\mathbb{Z}^{2}\setminus\left\{\bf 0\right\}\end{equation*}
\begin{equation*}\label{eq:0.26}
\mathrm{det}\binom{\textbf{n}+\textbf{i}}{\textbf{m}}b_{\textbf{n}}=\mathrm{det}\binom{\textbf{n}}{\textbf{m}+\textbf{i}}b_{\textbf{m}},
~\forall~\textbf{m},\textbf{n}\in\mathbb{Z}^{2}\setminus\left\{\bf 0\right\}.\end{equation*}
Adding the two identities above yields
\begin{equation}\label{eq:0.27}
4\mathrm{det}\binom{\textbf{n}}{\textbf{m}} b_{\textbf{m}+\textbf{n}}=\mathrm{det}\binom{\textbf{n}}{\textbf{m}+\textbf{i}}b_{\textbf{m}}+
\mathrm{det}\binom{\textbf{n}+\textbf{i}}{\textbf{m}}b_{\textbf{n}},
~\forall~\textbf{m},\textbf{n}\in\mathbb{Z}^{2}\setminus\left\{\bf 0\right\}.\end{equation}
On the one hand,  taking $\textbf{n}=(1,0)$, $m_{2}= 0$ in identity \eqref{eq:0.27} 
$$i_{2}\left(b_{(m_{1},0)}-m_{1}b_{\textbf{e}_{1}}\right)=0,~\forall  ~m_{1}\in\mathbb{Z}^{*}.$$
Since $i_{2}\not=0$, then
\begin{equation}\label{eq:0.28}
b_{(m_{1},0)}=m_{1}b_{\textbf{e}_{1}}, ~\forall  ~m_{1}\in\mathbb{Z}^{*}.\end{equation}
On the other hand,  taking $\textbf{n}=(0,1)$, $m_{1}= 0$ in identity \eqref{eq:0.27} 
$$i_{1}\left(b_{(0,m_{2})}-m_{2}b_{\textbf{e}_{2}}\right)=0,~\forall  ~m_{2}\in\mathbb{Z}^{*}.$$
Since $i_{1}\not=0$, then
\begin{equation}\label{eq:0.29}
b_{(0,m_{2})}=m_{2}b_{\textbf{e}_{2}},~\forall  ~m_{2}\in\mathbb{Z}^{*}.\end{equation}
Taking $n_{1} = 0$, $m_{2} = 0$ in identity \eqref{eq:0.27} 
\begin{equation*}
4m_{1}n_{2}b_{(m_{1},n_{2})}=n_{2}\left(m_{1}+i_{1}\right)b_{(m_{1},0)}+m_{1} \left(n_{2}+i_{2} \right)b_{(0,n_{2})},~\forall  ~m_{1},n_{2}\in\mathbb{Z}^{*}.\end{equation*}
Substituting identities \eqref{eq:0.28} and \eqref{eq:0.29} into the above identity, we get
\begin{equation*}
4m_{1}n_{2}b_{(m_{1},n_{2})}=n_{2}\left(m_{1}+i_{1}\right)m_{1}b_{\textbf{e}_{1}}+m_{1}\left(n_2+i_{2}\right)n_{2}b_{\textbf{e}_{2}},~\forall  ~m_{1},n_{2}\in\mathbb{Z}^{*}.\end{equation*}
Then 
\begin{equation}\label{eq:0.32}
4b_{(m_{1},n_{2})}=\left(m_{1}+i_{1}\right)b_{\textbf{e}_{1}}+\left(n_{2}+i_{2}\right)b_{\textbf{e}_{2}},~\forall  ~m_{1},n_{2}\in\mathbb{Z}^{*}.\end{equation}
Substituting identity \eqref{eq:0.32} into identity \eqref{eq:0.27}, then for those $\textbf{m},\textbf{n}\in\mathbb{Z}^{*}\times\mathbb{Z}^{*}$ such that ${\bf m+n}\in \mathbb{Z}^{*}\times\mathbb{Z}^{*}$, we get
\begin{eqnarray*}
&&4\mathrm{det}\binom{\textbf{n}}{\textbf{m}}\left(\left(m_{1}+n_{1}+i_{1}\right) b_{\textbf{e}_{1}}+\left(m_{2}+n_{2}+i_{2}\right)b_{\textbf{e}_{2}}\right)
\\
&=&\mathrm{det}\binom{\textbf{n}}{\textbf{m}+\textbf{i}}\left(\left(m_{1}+i_{1}\right)b_{\textbf{e}_{1}}+\left(m_{2}+i_{2}\right)b_{\textbf{e}_{2}}\right)
+\mathrm{det}\binom{\textbf{n}+\textbf{i}}{\textbf{m}}\left(\left(n_{1}+i_{1}\right)b_{\textbf{e}_{1}}+\left(n_{2}+i_{2}\right)b_{\textbf{e}_{2}}\right).\end{eqnarray*}
Particularly, taking $\textbf{m}=(-i_{1},i_{2})$, $\textbf{n}=(i_{1},-i_{2})$ and $\textbf{m}=(-i_{1},2i_{2})$, $\textbf{n}=(i_{1},-i_{2})$, respectively, we have
$$\left\{\begin{matrix}
i_{1}b_{\textbf{e}_{1}}+i_{2}b_{\textbf{e}_{2}}=0, \\
4i_{1}b_{\textbf{e}_{1}}+i_{2}b_{\textbf{e}_{2}}=0.
\end{matrix}\right.$$
Solving the system of linear equations above yields
$$b_{\textbf{e}_{1}}=b_{\textbf{e}_{2}}=0.$$
Substituting $b_{\textbf{e}_{1}}=0$ and $b_{\textbf{e}_{2}}=0$ into identity \eqref{eq:0.32}, we obtain
\begin{equation*}b_{\textbf{m}}=0,~\forall~ \textbf{m}\in\mathbb{Z}^{*}\times \mathbb{Z}^{*}.\end{equation*}
Substituting $b_{\textbf{e}_{1}}=0$ into identity \eqref{eq:0.28}, we get
\begin{equation*}b_{(m_{1},0)}=0,~\forall~ m_{1}\in\mathbb{Z}^{*}.\end{equation*}
Substituting $b_{\textbf{e}_{2}}=0$ into identity \eqref{eq:0.29}, we get
\begin{equation*}b_{(0,m_{2})}=0,~\forall~ m_{2}\in\mathbb{Z}^{*}.\end{equation*}
In summary, we proved
$$b_{\textbf{m}}=0,~\forall~\textbf{m}\in\mathbb{Z}^{2}\setminus\left\{\bf 0\right\}.$$

Combining all the analyses above, we deduce that $\varphi=0$.
\end{proof}

Secondly, we consider the even linear maps $\varphi: L(V)\to L(V)$,   which are characterized by their property of satisfying
$$\varphi\left( L(V)_{\bar{0}}\right)\subseteq L(V)_{\bar{0}},~~~\varphi\left( L(V)_{\bar{1}}\right)\subseteq L(V)_{\bar{1}}.$$
We thus have $\left | \varphi \right |=0$, and $\varphi$ is a $\frac 12$-superderivation of $L(V)$ if and only if
$$\varphi\left ( \left[x,y\right] \right ) =\frac{1}{2} \left(\left[\varphi\left ( x\right ),y\right]+\left[x,\varphi\left ( y\right )\right]\right), ~\forall~ x, y\in L(V)_{\bar 0}\cup L(V)_{\bar 1}.$$

\begin{theorem}\label{D}
 $\Delta _{\overline 0} \left (L(V) \right )=\left<{\rm id}\right>.$
\end{theorem}
\begin{proof}
Let $\varphi$ be an even $\frac 12$-superderivation of $L(V)$, then by identity \eqref{eq3.1}, we can write $\varphi=\sum\limits_{{\bf i}\in\mathbb Z^2}\varphi_{\bf i}$, where $\varphi_{\bf i}$ is also an even $\frac 12$-superderivation of $L(V)$ for all ${\bf i}\in \mathbb Z^2.$ 
For $\textbf{i}\in \mathbb{Z}^{2}$, since $\varphi\left( L(V)_{\bar{0}}\right)\subseteq L(V)_{\bar{0}}$, $\varphi\left( L(V)_{\bar{1}}\right)\subseteq L(V)_{\bar{1}}$, we suppose 
\begin{equation*}\varphi _{\textbf{i}}\left(L_{\textbf{m}}\right)=c_{\textbf{m}} L_{\textbf{m}+\textbf{i}},\end{equation*}
\begin{equation*}\varphi _{\textbf{i}}\left(G_{\textbf{m}}\right)=d_{\textbf{m}} G_{\textbf{m}+\textbf{i}}.\end{equation*}
where $\textbf{m}\in\mathbb{Z}^{2}\setminus\left\{\bf 0\right\}$. Applying $\varphi _{\textbf{i}}$ to identities \eqref{eq1.1} and \eqref{eq1.3}, we get $\forall~ {\bf m, n}\in\mathbb Z^2\setminus\{\bf 0\},$
\begin{equation}2\mathrm{det}\binom{\textbf{n}}{\textbf{m}} c_{\textbf{m}+\textbf{n}}=
\mathrm{det}\binom{\textbf{n}}{\textbf{m}+\textbf{i}}c_{\textbf{m}}+\mathrm{det}\binom{\textbf{n}+\textbf{i}}{\textbf{m}} c_{\textbf{n}},\label{eq:0.41}\end{equation}
\begin{equation}2c_{\textbf{n}+\textbf{m}}=d_{\textbf{n}}+d_{\textbf{m}}.\label{eq:0.42}\end{equation}
To determine the cofficients, we need to consider the following cases: 

\textbf{Case 1.} $\textbf{i}={\bf 0}$. 

By identity \eqref{eq:0.41} By identity \eqref{eq:0.4}, we get
\begin{equation}\mathrm{det}\binom{\textbf{n}}{\textbf{m}}\left(2c_{\textbf{m+n}}-c_{\textbf{m}}-c_{\textbf{n}}\right)=0,~\forall ~ \textbf{m},\textbf{n}\in\mathbb{Z}^{2}\setminus\left\{\bf 0\right\}.\label{eq:9.9}\end{equation}
Particularly, by taking $\textbf{n}= \textbf{e}_{1}$ and $\textbf{e}_{2}$ in identity \eqref{eq:9.9}, respectively, we have
\begin{equation}2c_{\textbf{m}+\textbf{e}_{1}}-c_{\textbf{m}}-c_{\textbf{e}_{1}}=0,~\forall~\textbf{m}\in\mathbb{Z}\times\mathbb{Z}^{*}.\label{eq:9.11}\end{equation}
\begin{equation}2c_{\textbf{m}+\textbf{e}_{2}}-c_{\textbf{m}}-c_{\textbf{e}_{2}}=0 ,~\forall~\textbf{m}\in\mathbb{Z}^{*}\times\mathbb{Z}.\label{eq:9.12}\end{equation}
 Fix $m_{2}\in\mathbb{Z}^{*}$ and treat $\left( c_{\textbf{m}}-c_{\textbf{e}_{1}}\right)_{m_{1}\in\mathbb{Z}}$ as a geometric sequence, then by identity \eqref{eq:9.11}, we have
\begin{equation} c_{\textbf{m}}=\left( c_{(1,m_{2})}-c_{\textbf{e}_{1}}\right)\left(\frac{1}{2}\right)^{m_{1}-1}
+c_{\textbf{e}_{1}},~\forall~\textbf{m}\in\mathbb{Z}\times\mathbb{Z}^{*}. \label{eq:9.13}\end{equation}
 Fix $m_{1}\in\mathbb{Z}^{*}$ and treat $\left(c_{\textbf{m}}-c_{\textbf{e}_{2}}\right)_{m_{2}\in\mathbb{Z}}$ as a geometric sequence, then by identity \eqref{eq:9.12},we have
\begin{equation*} c_{\textbf{m}}=\left( c_{(m_{1},0)}-c_{\textbf{e}_{2}}\right)\left(\frac{1}{2}\right)^{m_{2}}
+c_{\textbf{e}_{2}},~\forall~\textbf{m}\in\mathbb{Z}^{*}\times\mathbb{Z}. \label{eq:9.14}\end{equation*}
Taking $m_{1}=1$ in the above identity, we get
\begin{equation*} c_{(1,m_{2})}=\left(c_{\textbf{e}_{1}}-c_{\textbf{e}_{2}}\right)\left(\frac{1}{2}\right)^{m_{2}}+c_{\textbf{e}_{2}},~\forall~m_{2}\in\mathbb{Z}.\label{eq:9.15} \end{equation*}
By substituting the above identity into \eqref{eq:9.13}, we have
\begin{equation}c_{\textbf{m}}=\left(c_{\textbf{e}_{1}}-c_{\textbf{e}_{2}}\right)\left ( \left(\frac{1}{2}\right)^{m_{2}}-1 \right )\left ( \frac{1}{2}\right)^{m_{1}-1} +c_{\textbf{e}_{1}},~\forall~\textbf{m}\in \mathbb{Z}\times \mathbb{Z}^{*}.\label{eq:9.18}\end{equation}
Substituting ientity \eqref{eq:9.18} into \eqref{eq:9.9}, then for those $\textbf{m},\textbf{n}\in\mathbb{Z}\times\mathbb{Z}^{*}$ such that $\textbf{m}+\textbf{n}\in\mathbb{Z}\times\mathbb{Z}^{*}$, we get
\begin{eqnarray*}
\mathrm{det}\binom{\textbf{n}}{\textbf{m}}\left(c_{\textbf{e}_{1}}-c_{\textbf{e}_{2}}\right)&&
(\left(\left(\frac{1}{2}\right)^{m_{2}+n_{2}}-1\right)\left(\frac{1}{2}\right)^{m_{1}+n_{1}-2}-
\left(\left(\frac{1}{2}\right)^{m_{2}}-1\right)\left(\frac{1}{2}\right)^{m_{1}-1}\\
&&-\left(\left(\frac{1}{2}\right)^{n_{2}}-1\right)\left(\frac{1}{2}\right)^{n_{1}-1})=0.
\end{eqnarray*}
Setting $\textbf{m}=\left(2,1\right)$, $\textbf{n}=\left(1,1\right)$ in the identity above, we get
$$\frac{3}{8}\left(c_{\textbf{e}_{1}}-c_{\textbf{e}_{2}}\right)=0,$$
so
\begin{equation*}c_{\textbf{e}_{1}}=c_{\textbf{e}_{2}}.\label{eq:9.17}\end{equation*}
Then by identity \eqref{eq:9.18}, we have
\begin{equation}c_{\textbf{m}}=c_{\textbf{e}_{1}}=c_{\textbf{e}_{2}},~\forall~\textbf{m}\in \mathbb{Z}\times \mathbb{Z}^{*}.\end{equation}
Particularly,
$$c_{(m_{1},-1)}=c_{\textbf{e}_{2}},~\forall~ m_{1}\in \mathbb{Z}.$$
By setting $m_{2}=-1$ in identity \eqref{eq:9.12} and substituting the identity, we have
\begin{equation*}c_{(m_{1},0)}=c_{\textbf{e}_{2}}=c_{\textbf{e}_{1}},~\forall~ m_{1}\in \mathbb{Z}^{*}.\end{equation*}
Then we find that for all $\textbf{m}\in\mathbb{Z}^{2}\setminus\left\{\bf 0\right\}$, $c_{\textbf{m}}$ is equal to a constant, which we  denote by $c$. 
Then by taking $\textbf{n}=\textbf{m}$ in \eqref{eq:0.42}, we get $2c=2d_{\textbf{m}}$, $\forall {\bf m}\in\mathbb Z^2\setminus\{\bf 0\}.$ Then we know $\forall {\bf m}\in\mathbb Z^2\setminus\{{\bf 0}\}, d_{\textbf{n}}=c.$

\textbf{Case 2.} $\textbf{i}\in \left\{0\right\} \times \mathbb{Z}^{*}$ or $\mathbb{Z}^{*}\times\left\{0\right\} $.
 
 Without loss of generality, we suppose $\textbf{i}\in \left\{0\right\} \times \mathbb{Z}^{*}$.
 
 By identity \eqref{eq:0.41} and the same arguments as in \textbf{Case 2.} of Theorem \ref{A}, we know that $c_{\textbf{m}}=0$ for all ${\bf m}\in\mathbb Z^2\setminus\{{\bf 0}\}$. Then by taking $\textbf{n}=\textbf{m}$ in identity \eqref{eq:0.42}, we get $d_{\textbf{n}}=0$ for all ${\bf m}\in\mathbb Z^2\setminus\{{\bf 0}\}$.

\textbf{Case 3.} $\textbf{i}\in\mathbb{Z}^{*}\times\mathbb{Z}^{*}$. 

By identity \eqref{eq:0.41} and the same arguments as in \textbf{Case 3.} of Theorem \ref{A}, we get $c_{\textbf{m}}=0$ for all ${\bf m}\in\mathbb Z^2\setminus\{{\bf 0}\}$. Then by taking $\textbf{n}=\textbf{m}$ in identity \eqref{eq:0.42}, we get $d_{\textbf{n}}=0$ for all ${\bf m}\in\mathbb Z^2\setminus\{{\bf 0}\}.$ 

In summary, we have proven that there exists $c\in\mathbb C$ such that $\varphi=c{\rm id}$.
\end{proof}

By Theorem \ref{A}, Theorem \ref{D} and Lemma \ref{lem2.3}, we have

\begin{theorem}
There are no non-trivial transposed Poisson superalgebra structures defined on $(L(V), [\cdot, \cdot])$.
\end{theorem}

\subsection{Transposed $-1$-Poisson superalgebra structure on $L(V)$}
Firstly, we determine all the odd $-1$-superderivations of $L(V)$.

\begin{theorem}\label{A.}
$\Delta _{\overline 1} \left ( L(V) \right ) =\{0\}$.
\end{theorem}
\begin{proof} Let $\varphi$ be an odd -1-superderivation of $L(V)$, then by identity \eqref{eq3.1}, we can write $\varphi=\sum\limits_{{\bf i}\in\mathbb Z^2}\varphi_{\bf i}$, where $\varphi_{\bf i}$ is also an odd -1-superderivation of $L(V)$ for all ${\bf i}\in \mathbb Z^2.$ 
Let $\textbf{i}\in \mathbb{Z}^{2}$, due $\varphi_{\bf i}\left( L(V)_{\bar{0}}\right)\subseteq L(V)_{\bar{1}}$, $\varphi_{\bf i}\left( L(V)_{\bar{1}}\right)\subseteq L(V)_{\bar{0}}$, 
we suppose 
\begin{equation*}\varphi _{\textbf{i}}\left(L_{\textbf{m}}\right)=a_{\textbf{m}} G_{\textbf{m}+\textbf{i}},\end{equation*}
\begin{equation*}\varphi _{\textbf{i}}\left(G_{\textbf{m}}\right)=b_{\textbf{m}} L_{\textbf{m}+\textbf{i}}.\end{equation*}
for all $\textbf{m}\in\mathbb{Z}^{2}\setminus\left\{\bf 0\right\}$. Applying $\varphi _{\textbf{i}}$ to identities \eqref{eq1.1}-\eqref{eq1.3}, we obtain $\forall~ {\bf m, n}\in \mathbb Z^2\setminus\{\bf 0\},$
\begin{equation}\label{eq:0.4.}
\mathrm{det}\binom{\textbf{n}}{\textbf{m}} a_{\textbf{m}+\textbf{n}}=
-\mathrm{det}\binom{\textbf{n}}{\textbf{m}+\textbf{i}}a_{\textbf{m}}-\mathrm{det}\binom{\textbf{n}+\textbf{i}}{\textbf{m}} a_{\textbf{n}},
\end{equation}
\begin{equation}\label{eq:0.5.}
\mathrm{det}\binom{\textbf{n}}{\textbf{m}} b_{\textbf{m}+\textbf{n}}=-a_{\textbf{m}}-\mathrm{det}\binom{\textbf{n}+\textbf{i}}{\textbf{m}} b_{\textbf{n}},\end{equation}
\begin{equation}\label{eq:0.6.}
a_{\textbf{m}+\textbf{n}}=-
\mathrm{det}\binom{\textbf{n}}{\textbf{m}+\textbf{i}}b_{\textbf{m}}+\mathrm{det}\binom{\textbf{n}+\textbf{i}}{\textbf{m}} b_{\textbf{n}}.\end{equation}
 To determine the coefficients, we need to consider the following cases:

\textbf{Case 1.} $\textbf{i}={\bf 0}$.

 We know that for all ${\bf m}\in \mathbb Z^2\setminus\{\bf 0\},$ there exists an $n\in\mathbb Z^2\setminus\{\bf 0\},$ such that $\det\binom{\bf n}{\bf m}=0.$
 Thus by \eqref{eq:0.5.}, we obtain $\forall~ {\bf m}\in \mathbb Z^2\setminus\{{\bf 0}\}, a_{\bf m}=0. $  

Then the identities \eqref{eq:0.5.} and \eqref{eq:0.6.} become
\begin{equation}\mathrm{det}\binom{\textbf{n}}{\textbf{m}} \left (b_{\textbf{m}+\textbf{n}}+b_{\textbf{n}}\right)=0,\label{eq:0.7.}\end{equation}
\begin{equation}\mathrm{det}\binom{\textbf{n}}{\textbf{m}}\left (b_{\textbf{m}}-b_{\textbf{n}}\right)=0.\label{eq:0.8.}\end{equation}
Taking $\textbf{m}=\textbf{e}_{2} $, $\textbf{n}=\textbf{e}_{1}$ 
in identity \eqref{eq:0.8.}, we have
$b_{\textbf{e}_{1}}=b_{\textbf{e}_{2}}.$

Taking $\textbf{n}=\textbf{e}_{1}$ in identity \eqref{eq:0.8.}, we get
$$b_{\textbf{m}}=b_{\textbf{e}_{1}},~\forall~\textbf{m}\in\mathbb{Z}\times\mathbb{Z}^{*}. $$
Taking $\textbf{n}=\textbf{e}_{2}$ in identity \eqref{eq:0.8.}, we get
$$b_{\textbf{m}}=b_{\textbf{e}_{2}},~\forall~\textbf{m}\in\mathbb{Z}^{*}\times\mathbb{Z}. $$
Thus we proved that for all $\textbf{m}\in\mathbb{Z}^{2}\setminus\left\{{\bf 0}\right\}$, $b_{\textbf{m}}$ equals to a constant, by denoting this constant as $b$, then by \eqref{eq:0.7.}, we get 
\begin{equation*}2\mathrm{det}\binom{\textbf{n}}{\textbf{m}} b=0.\end{equation*}
Then we have
$$b=0.$$

\textbf{Case 2.} $\textbf{i}=(i_{1},i_{2})\in \left\{0\right\} \times \mathbb{Z}^{*}$ or $\mathbb{Z}^{*}\times\left\{0\right\} $. 

Without loss of generality, we suppose $\textbf{i}\in \left\{0\right\} \times \mathbb{Z}^{*}$. 

By identity \eqref{eq:0.4.}, we have $\forall~ \textbf{m},\textbf{n}\in\mathbb{Z}^{2}\setminus\left\{\bf 0\right\},$
\begin{equation}\label{eq:9.20.}
\mathrm{det}\binom{\textbf{n}}{\textbf{m}}a_{\textbf{m}+\textbf{n}}=-\left(\left(m_{2}+i_{2}\right)n_{1}-m_{1}n_{2}\right)a_{\textbf{m}}-
\left(m_{2}n_{1}-m_{1}\left(n_{2}+i_{2}\right)\right)a_{\textbf{n}}.\end{equation}
Taking $\textbf{n}= \textbf{e}_{1}$ and $\textbf{e}_{2}$ in identity \eqref{eq:9.20.}, we get
\begin{equation}
m_{2}a_{\textbf{m}+\textbf{e}_{1}}=-\left(m_{2}+i_{2}\right)a_{\textbf{m}}-\left(m_{2}-m_{1}i_{2}\right)a_{\textbf{e}_{1}},
~\forall~\textbf{m}\in\mathbb{Z}^{2}\setminus\left\{\bf 0\right\},\label{eq:9.21.}\end{equation}
\begin{equation}-m_{1}a_{\textbf{m}+\textbf{e}_{2}}=m_{1}a_{\textbf{m}}+m_{1}\left(1+i_{2}\right)a_{\textbf{e}_{2}},~\forall~  \textbf{m}\in\mathbb{Z}^{2}\setminus\left\{\bf 0\right\}.\label{eq:9.22.}\end{equation}
By taking $m_{2}= 0$ in identity \eqref{eq:9.21.}, we obtain
\begin{equation}a_{(m_{1},0)}=m_{1}a_{\textbf{e}_{1}},~~\forall  ~m_{1}\in\mathbb{Z}^{*}.\label{eq:9.24.}\end{equation}
By identity \eqref{eq:9.22.}, we have
\begin{equation}\label{eq:4.11}
a_{(m_{1},m_{2}+1)}=-a_{\textbf{m}}-\left(1+i_{2}\right)a_{\textbf{e}_{2}},~\forall~\textbf{m}\in\mathbb{Z}^{*}\times\mathbb{Z}.\end{equation}
Fix $m_1\in\mathbb Z^*,$ and treat $\left(a_{\textbf{m}}+\frac{1}{2}\left(1+i_{2}\right)a_{\textbf{e}_{2}}\right)_{m_{2}\in\mathbb{Z}}$ as a geometric sequence, then by identity \eqref{eq:4.11}, we have
\begin{equation*}a_{\textbf{m}}=\left( a_{(m_{1},0)}+\frac{1}{2}\left(1+i_{2}\right)a_{\textbf{e}_{2}}\right)\left(-1\right)^{m_{2}}-\frac{1}{2}\left(1+i_{2}\right)a_{\textbf{e}_{2}},
~\forall ~\textbf{m}\in\mathbb{Z}^{*}\times\mathbb{Z}.\label{eq:9.25.} \end{equation*}
Substituting identity \eqref{eq:9.24.} into the identity  above, we get
\begin{equation}a_{\textbf{m}}=m_{1}\left(-1\right)^{m_{2}}a_{\textbf{e}_{1}}+\frac{1}{2}\left(1+i_{2}\right)\left(\left(-1\right)^{m_{2}}-1\right)a_{\textbf{e}_{2}},
~\forall~  \textbf{m}\in\mathbb{Z}^{*}\times\mathbb{Z}.
\label{eq:9.26.}\end{equation}
Substituting identity \eqref{eq:9.26.} into identity \eqref{eq:9.20.}, for those $\textbf{m},\textbf{n}\in\mathbb{Z}^{*}\times\mathbb{Z}$ such that $\textbf{m}+\textbf{n}\in\mathbb{Z}^{*}\times\mathbb{Z}$,  we have
\begin{equation}\begin{split}\mathrm{det}\binom{\textbf{n}}{\textbf{m}}\left(\left(m_{1}+n_{1}\right)\left(-1\right)^{m_{2}+n_{2}} a_{\textbf{e}_{1}}+\frac{1}{2}
\left(1+i_{2}\right)\left(\left(-1\right)^{m_{2}+n_{2}}-1\right) a_{\textbf{e}_{2}}\right)\\=
\left(\left(m_{2}+i_{2}\right)n_{1}-m_{1}n_{2}\right)\left(m_{1} \left(-1\right)^{m_{2}}a_{\textbf{e}_{1}}+\frac{1}{2}\left(1+i_{2}\right)\left(\left(-1\right)^{m_{2}}-1\right) a_{\textbf{e}_{2}}\right)
\\+\left(m_{2}n_{1}-m_{1}\left(n_{2}+i_{2}\right)\right)\left(n_{1}\left(-1\right)^{n_{2}}a_{\textbf{e}_{1}}+\frac{1}{2}\left(1+i_{2}\right)\left(\left(-1\right)^{n_{2}}-1\right) a_{\textbf{e}_{2}}\right).\end{split}\label{eq:9.27.}\end{equation}

Now we need consider the two subcases:：$i_{2}= -1$ and $i_{2}\ne -1$.

\textbf{Subcase (i).} $i_{2}= -1$. Taking $\textbf{m}=(2,2)$, $\textbf{n}=(1,1)$ in identity \eqref{eq:9.27.}, we get
$a_{\textbf{e}_{1}}=0.$
Substituting $i_{2}= -1$ and $a_{\textbf{e}_{1}}=0$ into identity \eqref{eq:9.26.}, we get
$$a_{\textbf{m}}=0,~\forall~\textbf{m}\in\mathbb{Z}^{*}\times\mathbb{Z}.$$


\textbf{Subcase (ii).} $i_{2}\not= -1$. Taking $\textbf{m}=(1,1)$, $\textbf{n}=(2,2)$ and $\textbf{m}=(1,1)$, $\textbf{n}=(3, 3)$ in identity \eqref{eq:9.27.}, we get
$$\left\{\begin{matrix}
-2i_2(2a_{\textbf{e}_{1}}+\left(1+i_{2}\right)a_{\textbf{e}_{2}})=0,
 \\-2i_2\left(1+i_{2}\right)a_{\textbf{e}_{2}}=0.
\end{matrix}\right.$$
Since $i_2\neq 0$ and $1+i_2\neq 0$,  it follows that 
\begin{equation*}a_{\textbf{e}_{1}}=a_{\textbf{e}_{2}}=0.\label{eq:9.93.}\end{equation*}
Taking $a_{\textbf{e}_{1}}=0$ and $a_{\textbf{e}_{2}}=0$ in \eqref{eq:9.26.}, we get
$$a_{\textbf{m}}=0,~\forall~\textbf{m}\in\mathbb{Z}^{*}\times\mathbb{Z}.$$
Thus, in both of the two subcases, we proved $a_{\textbf{m}}=0,~\forall~\textbf{m}\in\mathbb{Z}^{*}\times\mathbb{Z}.$

Substituting $m_{1} = 1$, $\textbf{n}=(-1,0)$ in identity \eqref{eq:9.20.}, we get
$$a_{(0, m_2)}=0, ~\forall~ m_2\in\mathbb Z^{*}.$$

In summary, $$a_{\textbf{m}}=0,~\forall~\textbf{m}\in\mathbb{Z}^{2}\setminus\left\{\bf 0\right\}.$$
Thus, identities \eqref{eq:0.5.} and \eqref{eq:0.6.} transform into the following forms:
\begin{equation*}\label{eq:0.13.}
\mathrm{det}\binom{\textbf{n}}{\textbf{m}}b_{\textbf{m}+\textbf{n}}=-\left(m_{2}n_{1}-m_{1}\left(n_{2}+i_{2}\right)\right)b_{\textbf{n}},
~\forall~\textbf{m},\textbf{n}\in\mathbb{Z}^{2}\setminus\left\{\bf 0\right\}\end{equation*}
\begin{equation*}\label{eq:0.15.}
\left(m_{2}n_{1}-m_{1}\left(n_{2}+i_{2}\right)\right)b_{\textbf{n}}=\left(\left(m_{2}+i_{2}\right)n_{1}-m_{1}n_{2}\right)b_{\textbf{m}},
~\forall~\textbf{m},\textbf{n}\in\mathbb{Z}^{2}\setminus\left\{\bf 0\right\}.\end{equation*}
Subtracting the two identities above yields $\forall~\textbf{m},\textbf{n}\in\mathbb{Z}^{2}\setminus\left\{\bf 0\right\},$
\begin{equation}\label{eq:0.16.}
2\mathrm{det}\binom{\textbf{n}}{\textbf{m}}b_{\textbf{m}+\textbf{n}}=-
\left(\left(m_{2}+i_{2}\right)n_{1}-m_{1}n_{2}\right)b_{\textbf{m}}-\left(m_{2}n_{1}-m_{1}\left(n_{2}+i_{2}\right)\right)b_{\textbf{n}},
 \end{equation}
Taking $\textbf{n}= \textbf{e}_{1}$ and $\textbf{e}_{2}$ in identity \eqref{eq:0.16.}, respectively,  we get
\begin{equation}\label{eq:0.17.}
2m_{2}b_{\textbf{m}+\textbf{e}_{1}}=-\left(m_{2}+i_{2}\right)b_{\textbf{m}}-\left(m_{2}-m_{1}i_{2}\right)b_{\textbf{e}_{1}},
~\forall ~\textbf{m}\in\mathbb{Z}\times\mathbb{Z}^{*},\end{equation}
\begin{equation}\label{eq:0.18.}
2m_{1}b_{\textbf{m}+\textbf{e}_{2}}=-m_{1}b_{\textbf{m}}-m_{1}\left(1+i_{2}\right)b_{\textbf{e}_{2}},
~\forall ~\textbf{m}\in\mathbb{Z}^{*}\times\mathbb{Z}.\end{equation}
Taking $m_{2}= 0$ in identity \eqref{eq:0.17.}, we get
\begin{equation}b_{(m_{1},0)}=m_{1}b_{\textbf{e}_{1}},~\forall  ~m_{1}\in\mathbb{Z}^{*}.\label{eq:0.19.}\end{equation}
 Fix $m_1\in\mathbb Z^*$ and treat $\left(b_{\textbf{m}}+\left(\frac{1+i_{2}}{3}\right)b_{\textbf{e}_{2}}\right)_{m_{2}\in\mathbb{Z}}$ as a geometric sequence,  then by identity \eqref{eq:0.18.} we have
\begin{equation*}b_{\textbf{m}}=\left( b_{(m_{1},0)}+\left(\frac{1+i_{2}}{3}\right)b_{\textbf{e}_{2}}\right)\left(-\frac{1}{2}\right)^{m_{2}}-\left(\frac{1+i_{2}}{3}\right)b_{\textbf{e}_{2}},
~\forall ~\textbf{m}\in\mathbb{Z}^{*}\times\mathbb{Z}.\label{eq:0.21.} \end{equation*}
Substituting identity \eqref{eq:0.19.} into the identity above, we obtain
\begin{equation}\label{eq:0.21.}
b_{\textbf{m}}=m_{1}\left(-\frac{1}{2}\right)^{m_{2}}b_{\textbf{e}_{1}}+\left(\frac{1+i_{2}}{3}\right)\left(\left(-\frac{1}{2}\right)^{m_{2}}-1\right)b_{\textbf{e}_{2}},
~\forall ~ \textbf{m}\in\mathbb{Z}^{*}\times\mathbb{Z}.
\end{equation}
Substituting identity \eqref{eq:0.21.} into identity \eqref{eq:0.16.}, for those $\textbf{m},\textbf{n}\in\mathbb{Z}^{*}\times\mathbb{Z}$, such that $\textbf{m}+\textbf{n}\in\mathbb{Z}^{*}\times\mathbb{Z}$ ,  we get
\begin{equation}\begin{split}2\mathrm{det}\binom{\textbf{n}}{\textbf{m}}\left(\left(m_{1}+n_{1}\right)\left(-\frac{1}{2}\right)^{m_{2}+n_{2}} b_{\textbf{e}_{1}}+
\left(\frac{1+i_{2}}{3}\right)\left(\left(-\frac{1}{2}\right)^{m_{2}+n_{2}}-1\right) b_{\textbf{e}_{2}}\right)\\
=-
\left(\left(m_{2}+i_{2}\right)n_{1}-m_{1}n_{2}\right)\left(m_{1} \left(-\frac{1}{2}\right)^{m_{2}}b_{\textbf{e}_{1}}+\left(\frac{1+i_{2}}{3}\right)\left(\left(-\frac{1}{2}\right)^{m_{2}}-1\right) b_{\textbf{e}_{2}}\right)
\\-\left(m_{2}n_{1}-m_{1}\left(n_{2}+i_{2}\right)\right)\left(n_{1}\left(-\frac{1}{2}\right)^{n_{2}}b_{\textbf{e}_{1}}+\left(\frac{1+i_{2}}{3}\right)\left(\left(-\frac{1}{2}\right)^{n_{2}}-1\right) b_{\textbf{e}_{2}}\right).\end{split}\label{eq:0.22.}\end{equation}

Now we need to consider the two subcases: $i_{2}= -1$ and $i_{2}\ne -1$.

\textbf{Subcase (i).} $i_{2}= -1$. Taking $\textbf{m}=(2,2)$, $\textbf{n}=(1,1)$ in identity \eqref{eq:0.22.}, we have
$b_{\textbf{e}_{1}}=0.$
Substituting $i_{2}= -1$ and $b_{\textbf{e}_{1}}=0$ into identity \eqref{eq:0.21.}, we get
$$b_{\textbf{m}}=0,~\forall~\textbf{m}\in\mathbb{Z}^{*}\times\mathbb{Z}.$$

\textbf{Subcase (ii).} $i_{2}\not= -1$. Taking $\textbf{m}=(1,1)$, $\textbf{n}=(-2,-2)$ and $\textbf{m}=(1,1)$, $\textbf{n}=(2,2)$ in identity \eqref{eq:0.22.}, respectively, we obtain
$$\left\{\begin{matrix}
27b_{\textbf{e}_{1}}+8\left(1+i_{2}\right)b_{\textbf{e}_{2}}=0,
 \\2b_{\textbf{e}_{1}}+\left(1+i_{2}\right)b_{\textbf{e}_{2}}=0.
\end{matrix}\right.$$
Solving the system of linear equations yields
\begin{equation*}b_{\textbf{e}_{1}}=b_{\textbf{e}_{2}}=0.\end{equation*}
Substituting $b_{\textbf{e}_{1}}=0$ and $b_{\textbf{e}_{2}}=0$ into identity \eqref{eq:0.21.}, we have
$$b_{\textbf{m}}=0,~\forall~\textbf{m}\in\mathbb{Z}^{*}\times\mathbb{Z}.$$
Thus, in both of the two subcases we proved
$$b_{\textbf{m}}=0,~\forall~\textbf{m}\in\mathbb{Z}^{*}\times\mathbb{Z}.$$

Taking $n_{1} = -1$, $\textbf{m}=(1,0)$ in identity \eqref{eq:0.16.}, we get
$$b_{(0,n_{2})}=0,~\forall  ~n_{2}\in\mathbb{Z}^{*}.$$
In summary, we proved
$$b_{\textbf{m}}=0,~\forall~\textbf{m}\in\mathbb{Z}^{2}\setminus\left\{\bf 0\right\}.$$

\textbf{Case 3.} $\textbf{i}\in\mathbb{Z}^{*}\times\mathbb{Z}^{*}$. 

On the one hand,  taking $\textbf{n}=(1,0)$, $m_{2}= 0$ in identity \eqref{eq:0.4.}, we get
\begin{equation}a_{(m_{1},0)}=m_{1}a_{\textbf{e}_{1}}, ~\forall  ~m_{1}\in\mathbb{Z}^{*}.\label{eq:9.33.}\end{equation}
On the other hand, taking $\textbf{n}=(0,1)$, $m_{1}= 0$ in \eqref{eq:0.4.}, we have
\begin{equation}a_{(0,m_{2})}=m_{2}a_{\textbf{e}_{2}},~\forall  ~m_{2}\in\mathbb{Z}^{*}.\label{eq:9.34.}\end{equation}
Taking $n_{1} = 0$ and $m_{2} = 0$ in identity \eqref{eq:0.4.}, we have
\begin{equation}m_{1}n_{2}a_{(m_{1},n_{2})}=-n_{2}\left(m_{1}+i_{1}\right)a_{(m_{1},0)}-m_{1} \left(n_{2}+i_{2} \right)a_{(0,n_{2})},~\forall  ~m_{1},n_{2}\in\mathbb{Z}^{*}.\label{eq:9.35.}\end{equation}
Substituting identities \eqref{eq:9.33.} and\eqref{eq:9.34.} into identity \eqref{eq:9.35.}, we get
\begin{equation}m_{1}n_{2}a_{(m_{1},n_{2})}=-n_{2}\left(m_{1}+i_{1}\right)m_{1}a_{\textbf{e}_{1}}-m_{1}\left(n+i_{2}\right)n_{2}a_{\textbf{e}_{2}},~\forall  ~m_{1},n_{2}\in\mathbb{Z}^{*}.\end{equation}
Then
\begin{equation}a_{(m_{1},n_{2})}=-\left(m_{1}+i_{1})\right)a_{\textbf{e}_{1}}-\left(n_{2}+i_{2}\right)a_{\textbf{e}_{2}},~\forall  ~m_{1},n_{2}\in\mathbb{Z}^{*}.\label{eq:9.36.}\end{equation}
ubstituting \eqref{eq:9.36.} into \eqref{eq:0.4.}, for those $\textbf{m},\textbf{n}\in\mathbb{Z}^{*}\times\mathbb{Z}^{*}$ such that ${\bf m+n}\in\mathbb Z^*\times\mathbb Z^*$, we get
\begin{eqnarray*}
&&\mathrm{det}\binom{\textbf{n}}{\textbf{m}}\left(\left(m_{1}+n_{1}+i_{1}\right) a_{\textbf{e}_{1}}+\left(m_{2}+n_{2}+i_{2}\right)a_{\textbf{e}_{2}}\right)
\\
&=&-\mathrm{det}\binom{\textbf{n}}{\textbf{m}+\textbf{i}}\left(\left(m_{1}+i_{1}\right)a_{\textbf{e}_{1}}+\left(m_{2}+i_{2}\right)a_{\textbf{e}_{2}}\right)\\
&&-\mathrm{det}\binom{\textbf{n}+\textbf{i}}{\textbf{m}}\left(\left(n_{1}+i_{1}\right)a_{\textbf{e}_{1}}+\left(n_{2}+i_{2}\right)a_{\textbf{e}_{2}}\right).\end{eqnarray*}
Particularly, taking $\textbf{m}=(-2i_{1},i_{2})$, $\textbf{n}=(-i_{1},i_{2})$ and $\textbf{m}=(i_{1},-i_{2})$, $\textbf{n}=(-i_{1},i_{2})$, respectively, we get
$$\left\{\begin{matrix}
i_{1}a_{\textbf{e}_{1}}-9i_{2}a_{\textbf{e}_{2}}&=0, \\
i_{1}a_{\textbf{e}_{1}}+i_{2}a_{\textbf{e}_{2}}&=0.
\end{matrix}\right.$$
Hence, it immediately follows that
$$a_{\textbf{e}_{1}}=a_{\textbf{e}_{2}}=0.$$
Substituting $a_{\textbf{e}_{1}}=0$ and $a_{\textbf{e}_{2}}=0$ into identity \eqref{eq:9.36.}, we get
\begin{equation*}a_{\textbf{m}}=0,~\forall~ \textbf{m}\in\mathbb{Z}^{*}\times \mathbb{Z}^{*}.\end{equation*}
Substituting $a_{\textbf{e}_{1}}=0$ into identity \eqref{eq:9.33.}, we get
\begin{equation*}a_{(m_{1},0)}=0,~\forall~ m_{1}\in\mathbb{Z}^*.\end{equation*}
Substituting $a_{\textbf{e}_{2}}=0$ into identity \eqref{eq:9.34.}, we get
\begin{equation*}a_{(0,m_{2})}=0,~\forall ~m_{2}\in\mathbb{Z}^*.\end{equation*}
In summary, we proved
$$a_{\textbf{m}}=0,~\forall~\textbf{m}\in\mathbb{Z}^{2}\setminus\left\{\bf 0\right\}.$$
Then identities \eqref{eq:0.5.} and \eqref{eq:0.6.} transform into the following forms:
\begin{equation*}\label{eq:0.24.}
\mathrm{det}\binom{\textbf{n}}{\textbf{m}}b_{\textbf{m}+\textbf{n}}=-\mathrm{det}\binom{\textbf{n}+\textbf{i}}{\textbf{m}}b_{\textbf{n}},
~\forall~\textbf{m},\textbf{n}\in\mathbb{Z}^{2}\setminus\left\{\bf 0\right\}\end{equation*}
\begin{equation*}\label{eq:0.26.}
\mathrm{det}\binom{\textbf{n}+\textbf{i}}{\textbf{m}}b_{\textbf{n}}=\mathrm{det}\binom{\textbf{n}}{\textbf{m}+\textbf{i}}b_{\textbf{m}},
~\forall~\textbf{m},\textbf{n}\in\mathbb{Z}^{2}\setminus\left\{\bf 0\right\}.\end{equation*}
Subtracting the two identities above yields
\begin{equation}\label{eq:0.27.}
2\mathrm{det}\binom{\textbf{n}}{\textbf{m}} b_{\textbf{m}+\textbf{n}}=-\mathrm{det}\binom{\textbf{n}}{\textbf{m}+\textbf{i}}b_{\textbf{m}}-
\mathrm{det}\binom{\textbf{n}+\textbf{i}}{\textbf{m}}b_{\textbf{n}},
~\forall~\textbf{m},\textbf{n}\in\mathbb{Z}^{2}\setminus\left\{\bf 0\right\}.\end{equation}
On the one hand, taking $\textbf{n}=(1,0)$, $m_{2}= 0$ in identity \eqref{eq:0.27.}, we get  
\begin{equation}\label{eq:0.28.}
b_{(m_{1},0)}=m_{1}b_{\textbf{e}_{1}}, ~\forall  ~m_{1}\in\mathbb{Z}^{*}.\end{equation}
On the other hand,  taking $\textbf{n}=(0,1)$, $m_{1}= 0$ in identity \eqref{eq:0.27.}, we have  
\begin{equation}\label{eq:0.29.}
b_{(0,m_{2})}=m_{2}b_{\textbf{e}_{2}},~\forall  ~m_{2}\in\mathbb{Z}^{*}.\end{equation}
Taking $n_{1} = 0$, $m_{2} = 0$ in identity \eqref{eq:0.27.} 
\begin{equation}\label{eq:0.30.}
2m_{1}n_{2}b_{(m_{1},n_{2})}=-n_{2}\left(m_{1}+i_{1}\right)b_{(m_{1},0)}-m_{1} \left(n_{2}+i_{2} \right)b_{(0,n_{2})},~\forall  ~m_{1},n_{2}\in\mathbb{Z}^{*}.\end{equation}
Substituting identities \eqref{eq:0.28.} and \eqref{eq:0.29.} into identity \eqref{eq:0.30.}, we get
\begin{equation*}
2m_{1}n_{2}b_{(m_{1},n_{2})}=-n_{2}\left(m_{1}+i_{1}\right)m_{1}b_{\textbf{e}_{1}}-m_{1}\left(n_2+i_{2}\right)n_{2}b_{\textbf{e}_{2}},~\forall  ~m_{1},n_{2}\in\mathbb{Z}^{*}.\end{equation*}
Then
\begin{equation}\label{eq:0.32.}
2b_{(m_{1},n_{2})}=-\left(m_{1}+i_{1}\right)b_{\textbf{e}_{1}}-\left(n_{2}+i_{2}\right)b_{\textbf{e}_{2}},~\forall  ~m_{1},n_{2}\in\mathbb{Z}^{*}.\end{equation}
By substituting identity \eqref{eq:0.32.} into identity \eqref{eq:0.27.}, for those $\textbf{m},\textbf{n}\in\mathbb{Z}^{*}\times\mathbb{Z}^{*}$ such that ${\bf m+n}\in \mathbb{Z}^{*}\times\mathbb{Z}^{*}$, we get
\begin{eqnarray*}
&&2\mathrm{det}\binom{\textbf{n}}{\textbf{m}}\left(\left(m_{1}+n_{1}+i_{1}\right) b_{\textbf{e}_{1}}+\left(m_{2}+n_{2}+i_{2}\right)b_{\textbf{e}_{2}}\right)
\\
&=&-\mathrm{det}\binom{\textbf{n}}{\textbf{m}+\textbf{i}}\left(\left(m_{1}+i_{1}\right)b_{\textbf{e}_{1}}+\left(m_{2}+i_{2}\right)b_{\textbf{e}_{2}}\right)
-\mathrm{det}\binom{\textbf{n}+\textbf{i}}{\textbf{m}}\left(\left(n_{1}+i_{1}\right)b_{\textbf{e}_{1}}+\left(n_{2}+i_{2}\right)b_{\textbf{e}_{2}}\right).\end{eqnarray*}
Particularly, taking $\textbf{m}=(-i_{1},i_{2})$, $\textbf{n}=(i_{1},-i_{2})$ and $\textbf{m}=(-i_{1},2i_{2})$, $\textbf{n}=(i_{1},-i_{2})$, respectively, we have
$$\left\{\begin{matrix}
i_{1}b_{\textbf{e}_{1}}+i_{2}b_{\textbf{e}_{2}}=0, \\
10i_{1}b_{\textbf{e}_{1}}+13i_{2}b_{\textbf{e}_{2}}=0.
\end{matrix}\right.$$
Solving the system of linear equations above yields
$$b_{\textbf{e}_{1}}=b_{\textbf{e}_{2}}=0.$$
Substituting $b_{\textbf{e}_{1}}=0$ and $b_{\textbf{e}_{2}}=0$ into identity \eqref{eq:0.32.}, we obtain
\begin{equation*}b_{\textbf{m}}=0,~\forall~ \textbf{m}\in\mathbb{Z}^{*}\times \mathbb{Z}^{*}.\end{equation*}
Substituting $b_{\textbf{e}_{1}}=0$ into identity \eqref{eq:0.28.}, we get
\begin{equation*}b_{(m_{1},0)}=0,~\forall~ m_{1}\in\mathbb{Z}^{*}.\end{equation*}
Substituting $b_{\textbf{e}_{2}}=0$ into identity \eqref{eq:0.29.}, we get
\begin{equation*}b_{(0,m_{2})}=0,~\forall~ m_{2}\in\mathbb{Z}^{*}.\end{equation*}
In summary, we proved
$$b_{\textbf{m}}=0,~\forall~\textbf{m}\in\mathbb{Z}^{2}\setminus\left\{\bf 0\right\}.$$

Combining all the analyses above, we deduce that $\varphi=0$.
\end{proof}

Secondly, we determine all the even $-1$-superderivations of $L(V)$.
\begin{theorem}\label{D.}
 $\Delta _{\overline 0} \left (L(V) \right )=\{0\}.$
\end{theorem}
\begin{proof}
Let $\varphi$ be an even -1-superderivation of $L(V)$, then by identity \eqref{eq3.1}, we can write $\varphi=\sum\limits_{{\bf i}\in\mathbb Z^2}\varphi_{\bf i}$, where $\varphi_{\bf i}$ is also an even -1-superderivation of $L(V)$ for all ${\bf i}\in \mathbb Z^2.$ 
For $\textbf{i}\in \mathbb{Z}^{2}$, since $\varphi\left( L(V)_{\bar{0}}\right)\subseteq L(V)_{\bar{0}}$, $\varphi\left( L(V)_{\bar{1}}\right)\subseteq L(V)_{\bar{1}}$, we suppose 
\begin{equation*}\varphi _{\textbf{i}}\left(L_{\textbf{m}}\right)=c_{\textbf{m}} L_{\textbf{m}+\textbf{i}},\end{equation*}
\begin{equation*}\varphi _{\textbf{i}}\left(G_{\textbf{m}}\right)=d_{\textbf{m}} G_{\textbf{m}+\textbf{i}}.\end{equation*}
where $\textbf{m}\in\mathbb{Z}^{2}\setminus\left\{\bf 0\right\}$. Applying $\varphi _{\textbf{i}}$ to identities \eqref{eq1.1} and \eqref{eq1.3}, we get $\forall~ {\bf m, n}\in\mathbb Z^2\setminus\{\bf 0\},$
\begin{equation}\mathrm{det}\binom{\textbf{n}}{\textbf{m}} c_{\textbf{m}+\textbf{n}}=
-\mathrm{det}\binom{\textbf{n}}{\textbf{m}+\textbf{i}}c_{\textbf{m}}-\mathrm{det}\binom{\textbf{n}+\textbf{i}}{\textbf{m}} c_{\textbf{n}},\label{eq:0.41.8}\end{equation}
\begin{equation} c_{\textbf{m}+\textbf{n}}=
-d_{\textbf{m}}-d_{\textbf{n}},\label{eq:0.41..8}\end{equation}
To determine the cofficients, we need to consider the following cases: 

\textbf{Case 1.} $\textbf{i}={\bf 0}$. 
By identity \eqref{eq:0.41.8}, we get
\begin{equation}\mathrm{det}\binom{\textbf{n}}{\textbf{m}}\left(c_{\textbf{m+n}}+c_{\textbf{m}}+c_{\textbf{n}}\right)=0,~\forall ~ \textbf{m},\textbf{n}\in\mathbb{Z}^{2}\setminus\left\{\bf 0\right\}.\label{eq:9.98}\end{equation}
Particularly, taking $\textbf{n}= \textbf{e}_{1}$ and $\textbf{e}_{2}$ in identity \eqref{eq:9.98}, respectively, we have
\begin{equation}c_{\textbf{m}+\textbf{e}_{1}}+c_{\textbf{m}}+c_{\textbf{e}_{1}}=0,~\forall~\textbf{m}\in\mathbb{Z}\times\mathbb{Z}^{*}.\label{eq:9.118}\end{equation}
\begin{equation}c_{\textbf{m}+\textbf{e}_{2}}+c_{\textbf{m}}+c_{\textbf{e}_{2}}=0 ,~\forall~\textbf{m}\in\mathbb{Z}^{*}\times\mathbb{Z}.\label{eq:9.128}\end{equation}
 Fix $m_{2}\in\mathbb{Z}^{*}$ and treat $\left( c_{\textbf{m}}+\frac{1}{2}c_{\textbf{e}_{1}}\right)_{m_{1}\in\mathbb{Z}}$ as a geometric sequence, then by identity \eqref{eq:9.118}, we have
\begin{equation} c_{\textbf{m}}=\left( c_{(1,m_{2})}+\frac{1}{2}c_{\textbf{e}_{1}}\right)\left(-1\right)^{m_{1}-1}
-\frac{1}{2}c_{\textbf{e}_{1}},~\forall~\textbf{m}\in\mathbb{Z}\times\mathbb{Z}^{*}. \label{eq:9.138}\end{equation}
 Fix $m_{1}\in\mathbb{Z}^{*}$ and treat $\left(c_{\textbf{m}}+\frac{1}{2}c_{\textbf{e}_{2}}\right)_{m_{2}\in\mathbb{Z}}$ as a geometric sequence, then by identity \eqref{eq:9.128},we have
\begin{equation*} c_{\textbf{m}}=\left( c_{(m_{1},0)}+\frac{1}{2}c_{\textbf{e}_{2}}\right)\left(-1\right)^{m_{2}}
-\frac{1}{2}c_{\textbf{e}_{2}},~\forall~\textbf{m}\in\mathbb{Z}^{*}\times\mathbb{Z}. \label{eq:9.148}\end{equation*}
Taking $m_{1}=1$ in the above identity, we get
\begin{equation*} c_{(1,m_{2})}=\left(c_{\textbf{e}_{1}}+\frac{1}{2}c_{\textbf{e}_{2}}\right)\left(-1\right)^{m_{2}}-\frac{1}{2}c_{\textbf{e}_{2}},~\forall~m_{2}\in\mathbb{Z}.\label{eq:9.158} \end{equation*}
Substituting the above identity into identity \eqref{eq:9.138}, we have
\begin{equation}c_{\textbf{m}}=\left(\left ( \left(-1\right)^{m_{2}}+\frac{1}{2} \right )\left ( -1\right)^{m_{1}-1} -\frac{1}{2}\right)c_{\textbf{e}_{1}}+\left ( -1\right)^{m_{1}-1}\frac{1}{2}\left(\left ( -1\right)^{m_{2}}-1\right)c_{\textbf{e}_{2}},~\forall~\textbf{m}\in \mathbb{Z}\times \mathbb{Z}^{*}.\label{eq:9.188}\end{equation}
Substituting identity \eqref{eq:9.188} into \eqref{eq:9.98}, then for those $\textbf{m},\textbf{n}\in\mathbb{Z}\times\mathbb{Z}^{*}$ such that $\textbf{m}+\textbf{n}\in\mathbb{Z}\times\mathbb{Z}^{*}$, we get
\begin{eqnarray*}
\mathrm{det}\binom{\textbf{n}}{\textbf{m}}(&&\left( ( \left(-1\right)^{m_{2}+n_{2}}+\frac{1}{2} )\left ( -1\right)^{m_{1}+n_{1}-1} -\frac{1}{2}\right)c_{\textbf{e}_{1}}+\left ( -1\right)^{m_{1}+n_{1}-1}\frac{1}{2}\left(\left ( -1\right)^{m_{2}+n_{2}}-1\right)c_{\textbf{e}_{2}}\\
&&+\left( ( \left(-1\right)^{m_{2}}+\frac{1}{2}  )\left ( -1\right)^{m_{1}-1} -\frac{1}{2}\right)c_{\textbf{e}_{1}}+\left ( -1\right)^{m_{1}-1}\frac{1}{2}\left(\left ( -1\right)^{m_{2}}-1\right)c_{\textbf{e}_{2}}\\
&&+\left(( \left(-1\right)^{n_{2}}+\frac{1}{2} )\left ( -1\right)^{n_{1}-1} -\frac{1}{2}\right)c_{\textbf{e}_{1}}+\left ( -1\right)^{n_{1}-1}\frac{1}{2}\left(\left ( -1\right)^{n_{2}}-1\right)c_{\textbf{e}_{2}})=0.
\end{eqnarray*}
Particularly, taking $\textbf{m}=\left(1,3\right)$, $\textbf{n}=\left(5,7\right)$ and $\textbf{m}=\left(2,4\right)$, $\textbf{n}=\left(6,8\right)$ in the identity above, respectively, we have
$$\left\{\begin{matrix}
2c_{\textbf{e}_{1}}+c_{\textbf{e}_{2}}=0, \\
c_{\textbf{e}_{1}}=0.
\end{matrix}\right.$$
We get immediately
$c_{\textbf{e}_{1}}=c_{\textbf{e}_{2}}=0.$
Substituting $c_{\textbf{e}_{1}}=c_{\textbf{e}_{2}}=0$ into identity \eqref{eq:9.188}, we have
\begin{equation}c_{\textbf{m}}=0,~\forall~\textbf{m}\in \mathbb{Z}\times \mathbb{Z}^{*}.\end{equation}
Particularly,
$$c_{(m_{1},-1)}=c_{\textbf{e}_{2}}=0,~\forall~ m_{1}\in \mathbb{Z}.$$
Setting $m_{2}=-1$ in identity \eqref{eq:9.128} and substituting the identity above, we have
\begin{equation*}c_{(m_{1},0)}=0,~\forall~ m_{1}\in \mathbb{Z}^{*}.\end{equation*}
Then we find that for all $\textbf{m}\in\mathbb{Z}^{2}\setminus\left\{\bf 0\right\}$, $c_{\textbf{m}}=0$.
Substituting this result into identity \eqref{eq:0.41..8} and taking $\textbf{m}=\textbf{n}$, we get $d_{\textbf{n}}=0$ for all ${\bf m}\in\mathbb Z^2\setminus\{{\bf 0}\}.$

\textbf{Case 2.} $\textbf{i}\in \left\{0\right\} \times \mathbb{Z}^{*}$ or $\mathbb{Z}^{*}\times\left\{0\right\} $.
 
 Without loss of generality, we suppose $\textbf{i}\in \left\{0\right\} \times \mathbb{Z}^{*}$.
 
 By identity \eqref{eq:0.41.8} and the same arguments as in \textbf{Case 2.} of Theorem \ref{A.}, we know that $c_{\textbf{m}}=0$ for all ${\bf m}\in\mathbb Z^2\setminus\{{\bf 0}\}$. Substituting this result into identity \eqref{eq:0.41..8} and taking $\textbf{m}=\textbf{n}$, we get $d_{\textbf{n}}=0$ for all ${\bf n}\in\mathbb Z^2\setminus\{{\bf 0}\}$.

\textbf{Case 3.} $\textbf{i}\in\mathbb{Z}^{*}\times\mathbb{Z}^{*}$. 

By identity \eqref{eq:0.41.8} and the same arguments as in \textbf{Case 3.} of Theorem \ref{A.}, we get $c_{\textbf{m}}=0$ for all ${\bf m}\in\mathbb Z^2\setminus\{{\bf 0}\}$. Substituting this result into identity \eqref{eq:0.41..8} and taking $\textbf{m}=\textbf{n}$,  we get $d_{\textbf{n}}=0$ for all ${\bf n}\in\mathbb Z^2\setminus\{{\bf 0}\}.$ 

In summary, we have proven that $\varphi=0$.
\end{proof}

\begin{theorem}
   The Lie superalgebra $L(V)$ admits no non-trivial transposed $-1$-Poisson superalgebra structures. 
\end{theorem}

\subsection{Transposed $1$-Poisson superalgebra structure on $L(V)$}
For $\delta=1$, a $\delta$-superderivation of $L(V)$ is exactly a superderivation of $L(V)$.

\begin{lemma}
    The superderivaiton algebra ${\rm Der}(L(V))$ is $\mathbb Z^2$-graded.
\end{lemma}
\begin{proof}
    $L(V)$ is a finitely generated Lie superalgebra and has a natural  $\mathbb Z^2$-grading. Utilizing Lemma \ref{C},  ${\rm Der}(L(V))$ must be $\mathbb Z^2$-graded.
\end{proof}
Firstly, we determine all the odd superderivatioins of $L(V)$. 
\begin{lemma}\label{lem4.3}
    ${\rm Der(L(V))}_{\overline 1}={\rm ad}(L(V)_{\overline 1}).$
\end{lemma}
\begin{proof}
    Assuming $D_{\bf i}\in {\rm Der}(L(V))_{\overline 1}\cap {\rm Der}(L(V))_{\bf i}$ for some $\bf i\in \mathbb Z^2,$ we have 
    $$D_{\bf i}(L_{\bf m})=a_{\bf m}^{\bf i}G_{\bf m+i},$$
    $$D_{\bf i}(G_{\bf m})=b_{\bf m}^{\bf i}L_{\bf m+i}$$
    for all ${\bf m}\in\mathbb Z^2,$ where $a_{\bf m}^{\bf i}, b_{\bf m}^{\bf i}\in \mathbb C.$

    Given that
 \begin{center}
 {$D_{\bf i}([L_{\bf m}, L_{\bf n}])=[D_{\bf i}(L_{\bf m}), L_{\bf n}]+[L_{\bf m}, D_{\bf i}(L_{\bf n})]$~ for ${\bf m, n}\in \mathbb Z^2\setminus\{\bf 0\},$}
 \end{center}
    we obtain 
    \begin{equation}\label{eq4.2}
       \det\binom{\bf n}{\bf m}a_{\bf m+n}^{\bf i}=\det\binom{\bf n}{\bf m+i}a_{\bf m}^{\bf i}+\det\binom{\bf n+i}{\bf m}a_{\bf n}^{\bf i}.
    \end{equation}
Since \begin{center}
    $D_{\bf i}([L_{\bf m}, G_{\bf n}])=[D_{\bf i}(L_{\bf m}), G_{\bf n}]+[L_{\bf m}, D_{\bf i}(G_{\bf n})]$~ for  ${\bf m, n}\in \mathbb Z^2\setminus\{\bf 0\},$
\end{center}
we have  \begin{equation}\label{eq4.3}
       \det\binom{\bf n}{\bf m}b_{\bf m+n}^{\bf i}=a_{\bf m}^{\bf i}+\det\binom{\bf n+i}{\bf m}b_{\bf n}^{\bf i}.
    \end{equation}
Considering that 
\begin{center}
 {$D_{\bf i}([G_{\bf m}, G_{\bf n}])=[D_{\bf i}(G_{\bf m}), G_{\bf n}]-[G_{\bf m}, D_{\bf i}(G_{\bf n})]$~ for ${\bf m, n}\in \mathbb Z^2\setminus\{\bf 0\},$}
 \end{center}
 we deduce 
 \begin{equation}\label{eq4.4}
       a_{\bf m+n}^{\bf i}=\det\binom{\bf n}{\bf m+i}b_{\bf m}^{\bf i}-\det\binom{\bf n+i}{\bf m}b_{\bf n}^{\bf i}.
    \end{equation}
{\bf Case 1.} ${\bf i=0}.$ 

 By a similar arguments as for Case 1 in Therorem \ref{A}, we get $$a_{\bf m}^{\bf 0}=b_{\bf m}^{\bf 0}=0, ~\forall~ {\bf m}\in \mathbb Z^2\setminus\{0\}.$$
 Therefore, $D_{\bf 0}(L_{\bf m})=D_{\bf 0}(G_{\bf m})=0, ~\forall {\bf m}\in\mathbb Z^2\setminus\{{\bf 0}\}.$\\
 {\bf Case 2.} ${\bf i}\in \{0\}\times \mathbb Z^*$ or $\mathbb Z^*\times \{0\}.$ Without loss of genrality, we assume ${\bf i}\in \{0\}\times \mathbb Z^*.$

Taking ${\bf n=e_1, e_2}$ in identity \eqref{eq4.2} respecitively, we obtain 
\begin{equation}\label{eq4.5}
    m_2a_{\bf m+e_1}^{\bf i}=(m_2+i_2)a_{\bf m}^{\bf i}+(m_2-m_1i_2)a_{\bf e_1}^{\bf i}, ~\forall ~ {\bf m}\in \mathbb Z^2\setminus\{0\}
\end{equation}
\begin{equation}\label{eq4.6}
    a_{\bf m+e_2}^{\bf i}=a_{\bf m}^{\bf i}+(1+i_2)a_{\bf e_2}^{\bf i}, ~\forall ~ {\bf m}\in \mathbb Z^*\times \mathbb Z
\end{equation}
Setting $m_2=0$ in \eqref{eq4.5}, we get $i_2(a_{(m_1, 0)}^{\bf i}-m_1a_{\bf e_1}^{\bf i})=0, ~\forall ~m_1\in\mathbb Z^*,$ it follows that 
\begin{equation}\label{eq4.7}
    a_{(m_1, 0)}^{\bf i}=m_1a_{\bf e_1}^{\bf i}, ~\forall ~ m_1\in \mathbb Z^*
\end{equation}
By \eqref{eq4.6} we see that for a fixed $m_1\in\mathbb Z^*,$ $(a_{(m_1, m_2)})_{m_2\in \mathbb Z}$ is an arithmetic sequence with common difference $(1+i_2)a_{\bf e_2}^{\bf i}$,
so $a_{\bf m}^{\bf i}=a_{(m_1, 0)}^{\bf i}+m_2(1+i_2)a_{\bf e_2}^{\bf i}, ~\forall ~{\bf m}\in \mathbb Z^*\times \mathbb Z.$ Substituting \eqref{eq4.7}, we get 
\begin{equation}\label{eq4.8}
    a_{\bf m}^{\bf i}=m_1a_{\bf e_1}^{\bf i}+m_2(1+i_2)a_{\bf e_2}, ~\forall ~{\bf m}\in \mathbb Z^*\times \mathbb Z.
\end{equation}
{\bf subcase 1.} $i_2=-1.$

By \eqref{eq4.8}, we get $a_{\bf m}^{\bf i}=m_1a_{\bf e_1}^{\bf i},  ~\forall ~{\bf m}\in \mathbb Z^*\times \mathbb Z.$\\
{\bf subcase 2.} $i_2\neq -1.$

Substituting  \eqref{eq4.8} into \eqref{eq4.2}, then for all ${\bf m, n}\in\mathbb Z^*\times \mathbb Z$ such that ${\bf m+n}\in \mathbb Z^*\times \mathbb Z$, we have
\begin{eqnarray*}
    &&\det\binom{\bf n}{\bf m}((m_1+n_1)a_{\bf e_1}^{\bf i}+(m_2+n_2)(1+i_2)a_{\bf e_2}^{\bf i})\\
    &=&(\det\binom{\bf n}{\bf m}+i_2n_1)(m_1a_{\bf e_1}^{\bf i}+m_2(1+i_2)a_{\bf e_2}^{\bf i})+(\det\binom{\bf n}{\bf m}-m_1i_2)(n_1a_{\bf e_1}^{\bf i}+n_2(1+i_2)a_{\bf e_2}^{\bf i})
\end{eqnarray*}
Simplifying yields $i_2\det\binom{\bf n}{\bf m}(1+i_2)a_{\bf e_2}^{\bf i}=0.$ Since $i_2\neq 0$ or $-1$, we get $a_{\bf e_2}^{\bf i}=0.$  Then by \eqref{eq4.8}, we obtain $$a_{\bf m}^{\bf i}=m_1a_{\bf e_1}^{\bf i},  ~\forall ~{\bf m}\in \mathbb Z^*\times \mathbb Z.$$
Thus in both of the two subcases, we have $a_{\bf m}^{\bf i}=m_1a_{\bf e_1}^{\bf i},  ~\forall ~{\bf m}\in \mathbb Z^*\times \mathbb Z.$ 

Taking $m_1=1$ and ${\bf n}=(-1, 0)$ in identity \eqref{eq4.2}, we have $m_2a_{(0, m_2)}^{\bf i}=0$, ~$\forall~ m_2\in\mathbb Z^*.$ So $$a_{(0, m_2)}^{\bf i}=0,  ~\forall ~m_2\in \mathbb Z^*.$$
In summary,  \begin{equation}\label{eq4.9}
    \forall~ {\bf m}\in\mathbb Z^2\setminus\{{\bf 0}\}, ~  a_{\bf m}^{\bf i}=m_1a_{\bf e_1}^{\bf i}.
\end{equation}

By \eqref{eq4.3} and \eqref{eq4.9}, we have $\det\binom{\bf n}{\bf m}b_{\bf m+n}^{\bf i}-\det\binom{\bf n+i}{\bf m}b_{\bf n}^{\bf i}=m_1a_{\bf e_1}^{\bf i}.$ We konw that for all ${\bf n}\in \mathbb Z^*\times \mathbb Z, $ there exists an ${\bf m}\in \mathbb Z^*\times \mathbb Z$ such that $\det\binom{\bf n}{\bf m}=0,$ so we get $m_1i_2b_{\bf n}^{\bf i}=m_1a_{\bf e_1}^{\bf i}.$ 
This shows that \begin{equation}\label{eq4.10}
    \forall~ {\bf n}\in\mathbb Z^*\times \mathbb Z,~ b_{\bf n}=\frac{a_{\bf e_1}^{\bf i}}{i_2}.
\end{equation}
Taking ${\bf m}\in \{0\}\times \mathbb Z^*$ and ${\bf n}\in \mathbb Z^*\times \mathbb Z$ in identity \eqref{eq4.4}, we obtain $(m_2+i_2)b_{\bf m}^{\bf i}=a_{\bf e_1}^{\bf i}+m_2b_{\bf n}^{\bf i}.$ Substituting \eqref{eq4.10} yields $(m_2+i_2)b_{\bf m}^{\bf i}=(m_2+i_2)\frac{a_{\bf e_1}^{\bf i}}{i_2}.$ So  $$b_{(0, m_2)}^{\bf i}=\frac{a_{\bf e_1}^{\bf i}}{i_2}, ~\forall~ m_2\in\mathbb Z^*\setminus\{-i_2\}.$$
Taking ${\bf m}\in\mathbb Z^*\times \mathbb Z$ and ${\bf n}=(0, -i_2)$ in identity \eqref{eq4.3}, we have $$b_{(0, -i_2)}^{\bf i}=\frac{a_{\bf e_1}^{\bf i}}{i_2}.$$
In summary, $$b_{\bf m}^{\bf i}=\frac{a_{\bf e_1}^{\bf i}}{i_2}, ~\forall ~{\bf m}\in \mathbb Z^2\setminus\{{\bf 0}\}.$$ 
And we know that for all $ {\bf m}\in \mathbb Z^2\setminus\{{\bf 0}\},  b_{\bf m}^{\bf i}$ is constant, we denote it by $\lambda_{\bf i}$. Then $a_{\bf e_1}^{\bf i}=i_2\lambda_{\bf i}, $ and $$\forall ~{\bf m}\in\mathbb Z^2\setminus\{{\bf 0}\},~ a_{\bf m}^{\bf i}=\lambda_{\bf i} m_1i_2=\lambda_{\bf i}\det\binom{\bf m}{\bf i}.$$
Therefore, for all ${\bf m}\in \mathbb Z^2\setminus\{{\bf 0}\},$ 
$$D_{\bf i}(L_{\bf m})=\lambda_{\bf i}\det\binom{\bf m}{\bf i} G_{\bf m+i}=\lambda_{\bf i} {\rm ad}G_{\bf i}(L_{\bf m}), $$
$$D_{\bf i}(G_{\bf m})=\lambda_{\bf i} L_{\bf m+i}=\lambda_{\bf i}{\rm ad}G_{\bf i}(G_{\bf m}).$$
Thus, $D_{\bf i}=\lambda_{\bf i}{\rm ad}G_{\bf i}.$\\
{\bf Case 3.} ${\bf i}\in \mathbb Z^*\times \mathbb Z^*.$

Taking ${\bf n=e_1}$ in identity \eqref{eq4.2}, we have 
$$m_2a_{\bf m+e_1}^{\bf i}=(m_2+i_2)a_{\bf m}^{\bf i}+\left(m_2+\det\binom{\bf i}{\bf m}\right)a_{\bf e_1}^{\bf i}, ~\forall ~{\bf m}\in\mathbb Z^2\setminus\{{\bf 0}\}.$$
Setting $m_2=0,$ we obtain 
\begin{equation}\label{eq4.100}
    a_{(m_1, 0)}^{\bf i}=m_1a_{\bf e_1}^{\bf i}, ~\forall ~m_1\in\mathbb Z^*.
\end{equation}
Taking ${\bf n=e_2}$ in identity \eqref{eq4.2}, we have 
$$m_1a_{\bf m+e_2}^{\bf i}=(m_1+i_1)a_{\bf m}^{\bf i}+\left(m_1-\det\binom{\bf i}{\bf m}\right)a_{\bf e_2}^{\bf i}.$$
Setting $m_1=0$, we obtain 
\begin{equation}\label{eq4.111}
    a_{(0, m_2)}^{\bf i}=m_2a_{\bf e_2}^{\bf i}, ~\forall ~ m_2\in\mathbb Z^*.
\end{equation}
 Taking $n_1=0$ and $m_2=0$ in identity \eqref{eq4.2}, we get
 $$m_1n_2a_{(m_1, n_2)}^{\bf i}=(m_1n_2+i_1n_2)a_{(m_1, 0)}^{\bf i}+(m_1n_2+m_1i_2)a_{(0, n_2)}^{\bf i}, ~\forall~m_1, n_2\in\mathbb Z^*. $$
 Sustituting $a_{(m_1, 0)}^{\bf i}=m_1a_{\bf e_1}^{\bf i}$ and $a_{(0, m_2)}^{\bf i}=m_2a_{\bf e_2}^{\bf i}$ we get \begin{equation}\label{eq4.11}
     a_{(m_1, n_2)}^{\bf i}=(m_1+i_1)a_{\bf e_1}^{\bf i}+(n_2+i_2)a_{\bf e_2}^{\bf i}, ~\forall ~m_1, n_2\in \mathbb Z^*.
 \end{equation}
 Substituting \eqref{eq4.11} into \eqref{eq4.2}, for all ${\bf m, n}\in\mathbb Z^*\times\mathbb Z^*$ such that ${\bf m+n}\in\mathbb Z^*\times\mathbb Z^*,$ we have 
 \begin{eqnarray*}
    && \det\binom{\bf n}{\bf m}\left((m_1+n_1+i_1)a_{\bf e_1}^{\bf i}+(m_2+n_2+i_2)a_{\bf e_2}^{\bf i}\right)\\
    &=&\left(\det\binom{\bf n}{\bf m}+\det\binom{\bf n}{\bf i}\right)\left((m_1+i_1)a_{\bf e_1}^{\bf i}+(m_2+i_2)a_{\bf e_2}^{\bf i}\right)\\
    &&+\left(\det\binom{\bf n}{\bf m}+\det\binom{\bf i}{\bf m}\right)\left((n_1+i_1)a_{\bf e_1}^{\bf i}+(n_2+i_2)a_{\bf e_2}^{\bf i}\right)
 \end{eqnarray*}
 Setting ${\bf m}=(i_1, -i_2), {\bf n}=(-i_1, i_2)$ yields $i_1a_{\bf e_1}^{\bf i}+i_2a_{\bf e_2}^{\bf i}=0.$ Then by \eqref{eq4.11}, we obtain $$a_{\bf m}^{\bf i}=m_1a_{\bf e_1}^{\bf i}+m_2a_{\bf e_2}^{\bf i}, ~\forall ~{\bf m}\in \mathbb Z^*\times\mathbb Z^*.$$
 In summsary, $$a_{\bf m}^{\bf i}=m_1a_{\bf e_1}^{\bf i}+m_2a_{\bf e_2}^{\bf i}, ~\forall ~{\bf m}\in \mathbb Z^2\setminus\{{\bf 0}\}.$$
 Taking ${\bf m}={\bf e_1}$ in identity \eqref{eq4.3}, we have 
 $$-n_2b_{(n_1+1, n_2)}^{\bf i}=a_{\bf e_1}^{\bf i}-(n_2+i_2)b_{(n_1, n_2)}^{\bf i}, ~\forall ~{\bf n}\in\mathbb Z^2\setminus\{{\bf 0}\}.$$
 Setting $n_2=0$ yields $i_2b_{(n_1, 0)}^{\bf i}=a_{\bf e_1}^{\bf i},$ so 
\begin{equation}\label{eq4.13}
    b_{(n_1, 0)}^{\bf i}=\frac{a_{\bf e_1}^{\bf i}}{i_2}, ~\forall ~n_1\in\mathbb Z^*.
\end{equation}
 Taking ${\bf m}={\bf e_2}$ in identity \eqref{eq4.3}, we have 
 $$n_1b_{(n_1, n_2+1)}^{\bf i}=a_{\bf e_2}^{\bf i}+(n_1+i_1)b_{(n_1, n_2)}^{\bf i}, ~\forall ~{\bf n}\in\mathbb Z^2\setminus\{{\bf 0}\}.$$
  Setting $n_1=0$ yields $i_1b_{(0, n_2)}^{\bf i}=-a_{\bf e_1}^{\bf i},$ so 
 \begin{equation}\label{eq4.14}
     b_{(0, n_2)}^{\bf i}=-\frac{a_{\bf e_2}^{\bf i}}{i_1}=\frac{a_{\bf e_1}^{\bf i}}{i_2}, ~\forall~n_2\in\mathbb Z^*.
 \end{equation}
 Taking ${\bf m}=(n_1, 0)$ and ${\bf n}=(0, n_2)$ in identity\eqref{eq4.3}, and substituting \eqref{eq4.100} and \eqref{eq4.14}, we deduce 
 $$b_{\bf n}^{\bf i}=\frac{a_{\bf e_1}^{\bf i}}{i_2}, ~\forall ~{\bf n}\in\mathbb Z^*\times\mathbb Z^*.$$
 Therefore, for all ${\bf n}\in\mathbb Z^2\setminus\{{\bf 0}\},$ $b_{\bf n}^{\bf i}$ is constant, we denote it by $\lambda_{\bf i}'$. Then by \eqref{eq4.100} and \eqref{eq4.111}, we obtain $$a_{(n_1, 0)}^{\bf i}=n_1i_2\lambda_{\bf i}'=\det\binom{\bf n}{\bf i}\lambda_{\bf i}', ~\forall ~n_1\in\mathbb Z^*.$$
 $$a_{(0, n_2)}^{\bf i}=-n_2i_1\lambda_{\bf i}'=\det\binom{\bf n}{\bf i}\lambda_{\bf i}', ~\forall ~n_1\in\mathbb Z^*.$$
 Taking ${\bf m}=(n_1, 0)$ and ${\bf n}=(0, n_2)$ in identity\eqref{eq4.4}, and substituting \eqref{eq4.13} and \eqref{eq4.14}, we deduce 
 $$a_{(n_1, n_2)}^{\bf i}=(n_1i_2-i_1n_2)\lambda_{\bf i}'=\det\binom{\bf n}{\bf i}\lambda_{\bf i}', ~\forall~{\bf n}\in\mathbb Z^*\times \mathbb Z^*.$$
 Therefore, for all ${\bf m}\in\mathbb Z^2\setminus\{{\bf 0}\}$,
 $$D_{\bf i}(L_{\bf m})=\lambda_{\bf i}'\det\binom{\bf m}{\bf i}G_{\bf m+i}=\lambda_{\bf i}'{\rm ad}G_{\bf i}(L_{\bf m}),$$
 $$D_{\bf i}(G_{\bf m})=\lambda_{\bf i}'L_{\bf m+i}=\lambda_{\bf i}'{\rm ad}G_{\bf i}(G_{\bf m}).$$
 Thus, $D_{\bf i}=\lambda_{\bf i}'{\rm ad}G_{\bf i}.$
\end{proof}
\begin{lemma}\label{lem4.4}
    ${\rm Der}(L(V))_{\overline 0}={\rm ad}(L(V))_{\overline 0}+\mathbb C d_1+\mathbb C d_2,$ where $d_1$ and $d_2$ are defined as follow:
 \begin{eqnarray*}
      \forall ~ {\bf m}\in \mathbb Z^2\setminus\{{\bf 0}\}, ~&&d_1(L_{\bf m})=m_1L_{\bf m}, ~ d_2(L_{\bf m})=m_2L_{\bf m},\\
 &&d_1(G_{\bf m})=m_1G_{\bf m}, ~ d_2(G_{\bf m})=m_2G_{\bf m}.
 \end{eqnarray*}
\end{lemma}
\begin{proof}
    Assuming $D_{\bf i}\in {\rm Der}(L(V))_{\overline 0}\cap {\rm Der}(L(V))_{\bf i}$ for some $\bf i\in \mathbb Z^2,$ we have 
    $$D_{\bf i}(L_{\bf m})=c_{\bf m}^{\bf i}L_{\bf m+i},$$
    $$D_{\bf i}(G_{\bf m})=d_{\bf m}^{\bf i}G_{\bf m+i}$$
    for all ${\bf m}\in\mathbb Z^2,$ where $c_{\bf m}^{\bf i}, d_{\bf m}^{\bf i}\in \mathbb C.$

    Given that
 \begin{center}
 {$D_{\bf i}([L_{\bf m}, L_{\bf n}])=[D_{\bf i}(L_{\bf m}), L_{\bf n}]+[L_{\bf m}, D_{\bf i}(L_{\bf n})]$~ for ${\bf m, n}\in \mathbb Z^2\setminus\{\bf 0\},$}
 \end{center}
    we obtain 
    \begin{equation}\label{eq4.16}
       \det\binom{\bf n}{\bf m}c_{\bf m+n}^{\bf i}=\det\binom{\bf n}{\bf m+i}c_{\bf m}^{\bf i}+\det\binom{\bf n+i}{\bf m}c_{\bf n}^{\bf i}.
    \end{equation}
Since \begin{center}
    $D_{\bf i}([L_{\bf m}, G_{\bf n}])=[D_{\bf i}(L_{\bf m}), G_{\bf n}]+[L_{\bf m}, D_{\bf i}(G_{\bf n})]$~ for  ${\bf m, n}\in \mathbb Z^2\setminus\{\bf 0\},$
\end{center}
we have  \begin{equation}\label{eq4.17}
       \det\binom{\bf n}{\bf m}d_{\bf m+n}^{\bf i}=\det\binom{\bf n}{\bf m+i}c_{\bf m}^{\bf i}+\det\binom{\bf n+i}{\bf m}d_{\bf n}^{\bf i}.
    \end{equation}
Considering that 
\begin{center}
 {$D_{\bf i}([G_{\bf m}, G_{\bf n}])=[D_{\bf i}(G_{\bf m}), G_{\bf n}]+[G_{\bf m}, D_{\bf i}(G_{\bf n})]$~ for ${\bf m, n}\in \mathbb Z^2\setminus\{\bf 0\},$}
 \end{center}
 we deduce 
 \begin{equation}\label{eq4.18}
       c_{\bf m+n}^{\bf i}=d_{\bf m}^{\bf i}+d_{\bf n}^{\bf i}.
    \end{equation}
{\bf Case 1.} ${\bf i=0}.$ 

Taking ${\bf n=e_1}$ in identity \eqref{eq4.16}, we get 
$$c_{(m_1+1, m_2)}^{\bf 0}=c_{\bf m}^{\bf 0}+c_{\bf e_1}^{\bf 0}, ~\forall ~{\bf m}\in \mathbb Z\times \mathbb Z^*$$
Then we see that for a fixed $m_2\in\mathbb Z^*,$ $(c_{(m_1, m_2)})_{m_1\in \mathbb Z}$ is an arithmetic sequence with common difference $c_{\bf e_1}^{\bf 0}$,
so \begin{equation}\label{eq4.19}
    c_{\bf m}^{\bf 0}=c_{(0, m_2)}^{\bf 0}+m_1c_{\bf e_1}^{\bf 0}, ~\forall ~{\bf m}\in \mathbb Z\times \mathbb Z^*.
\end{equation}
Similarly, taking ${\bf n=e_2}$ in identity \eqref{eq4.16}, we deduce 
\begin{equation}\label{eq4.20}
   c_{\bf m}^{\bf 0}=c_{(m_1, 0)}^{\bf 0}+m_2c_{\bf e_2}^{\bf 0}, ~\forall ~{\bf m}\in \mathbb Z^*\times \mathbb Z.
\end{equation}
Taking ${\bf m}=(m_1, 0)$ and ${\bf n}=(0, m_2)$ in identity \eqref{eq4.16}, we have 
$$c_{(m_1, m_2)}^{\bf 0}=c_{(m_1, 0)}^{\bf 0 }+c_{(0, m_2)}^{\bf 0}, ~\forall~(m_1, m_2)\in\mathbb Z^*\times\mathbb Z^*.$$
Then combining with identities \eqref{eq4.19} and \eqref{eq4.20}, we obtain 
\begin{equation}\label{eq4.21}
    c_{(m_1, 0)}^{\bf 0}=m_1c_{\bf e_1}^{\bf 0}, ~\forall ~m_1\in\mathbb Z^*,
\end{equation}
\begin{equation}\label{eq4.22}
    c_{(0, m_2)}^{\bf 0}=m_2c_{\bf e_2}^{\bf 0}, ~\forall ~m_2\in\mathbb Z^*.
\end{equation}
And then we get 

    $$c_{(m_1, m_2)}^{\bf 0}=m_1c_{\bf e_1}^{\bf 0}+m_2c_{\bf e_2}^{\bf 0}, ~\forall~(m_1, m_2)\in\mathbb Z^*\times\mathbb Z^*.$$
Thus,   \begin{equation}\label{eq4.23}  
c_{(m_1, m_2)}^{\bf 0}=m_1c_{\bf e_1}^{\bf 0}+m_2c_{\bf e_2}^{\bf 0}, ~\forall~(m_1, m_2)\in\mathbb Z^2\setminus\{{\bf 0}\}.
\end{equation}
Taking ${\bf n=e_1}$ and ${\bf n=e_2}$ in identity \eqref{eq4.17} respectively, we get 
\begin{equation}\label{eq4.24}
     d_{(m_1, m_2)}^{\bf 0}=(m_1-1)c_{\bf e_1}^{\bf 0}+m_2c_{\bf e_2}^{\bf 0}+d_{\bf e_1}^{\bf 0}, ~\forall~ (m_1, m_2)\in\mathbb Z\times \mathbb Z^*,
\end{equation}
\begin{equation}\label{eq4.25}
     d_{(m_1, m_2)}^{\bf 0}=m_1c_{\bf e_1}^{\bf 0}+(m_2-1)c_{\bf e_2}^{\bf 0}+d_{\bf e_2}^{\bf 0}, ~\forall~ (m_1, m_2)\in\mathbb Z^*\times \mathbb Z.
\end{equation}
Taking ${\bf m}=(m_1, 0), {\bf n}=(0, m_2)$ in identity \eqref{eq4.17} and substututing \eqref{eq4.21}, we have 
$$d_{(m_1, m_2)}^{\bf 0}=m_1c_{\bf e_1}^{\bf 0}+d_{(0, m_2)}^{\bf 0}, ~\forall ~(m_1, m_2)\in\mathbb Z^*\times \mathbb Z^*.$$
Combining with \eqref{eq4.24} and \eqref{eq4.25}, we get 
\begin{equation}\label{eq4.26}
    d_{(0, m_2)}^{\bf 0}=m_2c_{\bf e_2}^{\bf 0}+d_{\bf e_1}^{\bf 0}-c_{\bf e_1}^{\bf 0}=m_2c_{\bf e_2}^{\bf 0}+d_{\bf e_2}^{\bf 0}-c_{\bf e_2}^{\bf 0}, ~\forall ~m_2\in\mathbb Z^*.
\end{equation}
So 
    \begin{equation}\label{eq4.27}
        d_{\bf e_1}^{\bf 0}-c_{\bf e_1}^{\bf 0}=d_{\bf e_2}^{\bf 0}-c_{\bf e_2}^{\bf 0}.
    \end{equation}
Similarly, taking ${\bf m}=(0, m_2), {\bf n}=(m_1, 0)$ in identity \eqref{eq4.17}, we deduce 
\begin{equation}\label{eq4.28}
d_{(m_1, 0)}^{\bf 0}=m_1c_{\bf e_1}^{\bf 0}+d_{\bf e_1}^{\bf 0}-c_{\bf e_1}^{\bf 0}=m_1c_{\bf e_1}^{\bf 0}+d_{\bf e_2}^{\bf 0}-c_{\bf e_2}^{\bf 0}, ~\forall~m_1\in\mathbb Z^*.
\end{equation}
Taking ${\bf m}=(m_1, 0), {\bf n}=(0, m_2)$ in identity \eqref{eq4.18} and substituting \eqref{eq4.23}, \eqref{eq4.26}, \eqref{eq4.28},  we obtain
$$m_1c_{\bf e_1}^{\bf 0}+m_2c_{\bf e_2}^{\bf 0}=m_1c_{\bf e_1}^{\bf 0}+d_{\bf e_1}^{\bf 0}-c_{\bf e_1}^{\bf 0}+m_2c_{\bf e_2}^{\bf 0}+d_{\bf e_1}^{\bf 0}-c_{\bf e_1}^{\bf 0}, ~\forall ~(m_1, m_2)\in\mathbb Z^*\times \mathbb Z^*.$$
Then we get $d_{\bf e_1}^{\bf 0}=c_{\bf e_1}^{\bf 0},$ and $d_{\bf e_2}^{\bf 0}=c_{\bf e_2}^{\bf 0}.$ 

Thus, $$d_{(m_1, m_2)}^{\bf 0}=m_1c_{\bf e_1}^{\bf 0}+m_2c_{\bf e2}^{\bf 0}, ~\forall ~(m_1, m_2)\in \mathbb Z^2\setminus\{{\bf 0}\}.$$

Therefore, for all ${\bf m}\in\mathbb Z^2\setminus\{{\bf 0}\}$,  $$D_{\bf 0}(L_{\bf m})=(m_1c_{\bf e_1}^{\bf 0}+m_2c_{\bf e_2}^{\bf 0})L_{\bf m}=(c_{\bf e_1}^{\bf 0}d_1+c_{\bf e_2}^{\bf 0}d_2)(L_{\bf m}),$$ 
$$D_{\bf 0}(G_{\bf m})=(m_1c_{\bf e_1}^{\bf 0}+m_2c_{\bf e_2}^{\bf 0})G_{\bf m}=(c_{\bf e_1}^{\bf 0}d_1+c_{\bf e_2}^{\bf 0}d_2)(G_{\bf m}).$$
Thus, $$D_{\bf 0}=c_{\bf e_1}^{\bf 0}d_1+c_{\bf e_2}^{\bf 0}d_2. $$
{\bf Case 2.} ${\bf i}\in \{0\}\times \mathbb Z^*$ or $\mathbb Z^*\times\{0\}.$ Without loss of genrality, we assume ${\bf i}\in \{0\}\times \mathbb Z^*.$

By the same arguments as for {\bf Case 2} in Lemma \ref{lem4.3}, we obtain 
\begin{equation}\label{eq4.29}
{\bf m}\in \mathbb Z^2\setminus\{{\bf 0}\}, c_{\bf m}=m_1c_{\bf e_1}^{\bf i} 
\end{equation} 
Taking ${\bf n=e_1}$ in identity \eqref{eq4.17} and substituting \eqref{eq4.29}, we have 
$$m_2d_{(m_1+1, m_2)}^{\bf i}=m_1m_2c_{\bf e_1}^{\bf i}+m_2d_{\bf e_1}^{\bf i}+m_1i_2(c_{\bf e_1}^{\bf i}-d_{\bf e_1}^{\bf i}), ~\forall~ {\bf m}\in \mathbb Z^2\setminus\{{\bf 0}\}.$$
Setting $m_2=0, m_1\neq 0,$ we get $c_{\bf e_1}^{\bf i}=d_{\bf e_1}^{\bf i}.$ Then 
$$\forall {\bf m}\in\mathbb Z\times \mathbb Z^*, ~d_{\bf m}^{\bf i}=m_1c_{\bf e_1}^{\bf i}.$$
Taking ${\bf n=e_2}$ in identity \eqref{eq4.17} and substituting \eqref{eq4.29}, we have 
$$d_{(m_1, m_2)}^{\bf i}=d_{(m_1, m_2+1)}^{\bf i}=m_1c_{\bf e_1}^{\bf i}+(i_2+1)d_{\bf e_2}^{\bf i}, ~\forall ~(m_1, m_2)\in\mathbb Z^*\times \mathbb Z.$$
Comparing the two identities above, we obtain $(i_2+1)d_{\bf e_2}^{\bf i}=0.$ So 
$$\forall ~{\bf m}\in \mathbb Z^*\times\mathbb Z,~ d_{\bf m}^{\bf i}=m_1c_{\bf e_1}^{\bf i}.$$
Thus, $$\forall ~{\bf m}\in \mathbb Z^2\setminus\{{\bf 0}\},~ d_{\bf m}^{\bf i}=m_1c_{\bf e_1}^{\bf i}.$$
Therefore, for all ${\bf m}\in \mathbb Z^2\setminus\{{\bf 0}\},$
$$D_{\bf i}(L_{\bf m})=m_1c_{\bf e_1}^{\bf i}L_{\bf m+i}=\frac{c_{\bf e_1}^{\bf i}}{i_2}\det\binom{\bf m}{\bf i}L_{\bf m+i}=\frac{c_{\bf e_1}^{\bf i}}{i_2}{\rm ad}L_{\bf i}(L_{\bf m}),$$
$$D_{\bf i}(G_{\bf m})=m_1c_{\bf e_1}^{\bf i}G_{\bf m+i}=\frac{c_{\bf e_1}^{\bf i}}{i_2}\det\binom{\bf m}{\bf i}G_{\bf m+i}=\frac{c_{\bf e_1}^{\bf i}}{i_2}{\rm ad}L_{\bf i}(G_{\bf m}).$$
Thus, $$D_{\bf i}=\dfrac{c_{\bf e_1}^{\bf i}}{i_2}{\rm ad}L_{\bf i}.$$
{\bf Case 3.} ${\bf i}\in\mathbb Z^*\times \mathbb Z^*.$

By the same arguments as for {\bf Case 3} in Lemma \ref{lem4.3}, we obtain 
\begin{equation}\label{eq.60}
    c_{\bf m}^{\bf i}=m_1c_{\bf e_1}^{\bf i}+m_2c_{\bf e_2}^{\bf i}, ~\forall ~{\bf m}\in\mathbb Z^2\setminus\{{\bf 0}\}.
\end{equation}
Taking ${\bf n=e_1}$ in identity \eqref{eq4.17}, we have
\begin{equation}\label{eq4.61}
    m_2d_{(m_1+1, m_2)}^{\bf i}=(m_2+i_2)c_{\bf m}^{\bf i}+(m_2(1+i_1)-m_1i_2)d_{\bf e_1}^{\bf i}, ~\forall~(m_1, m_2)\in\mathbb Z^2\setminus\{{\bf 0}\}.
\end{equation}
 Setting $m_2=0$ yields 
$i_2m_1(c_{\bf e_1}^{\bf i}-d_{\bf e_1}^{\bf i})=0, ~\forall~ m_1\in\mathbb Z^*.$
So $c_{\bf e_1}^{\bf i}=d_{\bf e_1}^{\bf i}.$ Substituting this identity into \eqref{eq4.61}, we get \begin{equation}\label{eq4.62}
    d_{(m_1, m_2)}^{\bf i}=(m_1+i_1)c_{\bf e_1}^{\bf i}+(m_2+i_2)c_{\bf e_2}^{\bf i}, ~\forall ~(m_1, m_2)\in\mathbb Z\times\mathbb Z^*.
\end{equation}
Taking ${\bf n=e_2}$ in identity \eqref{eq4.17}, we have
\begin{equation}\label{eq4.63}
    m_1d_{(m_1, m_2+1)}^{\bf i}=(m_1+i_1)c_{\bf m}^{\bf i}+(m_1(1+i_2)-m_2i_1)d_{\bf e_2}^{\bf i}, ~\forall~(m_1, m_2)\in\mathbb Z^2\setminus\{{\bf 0}\}.
\end{equation}
Setting $m_1=0$ yields $i_1m_2(c_{\bf e_2}^{\bf i}-d_{\bf e_2}^{\bf i})=0, ~\forall~ m_2\in\mathbb Z^*.$ So $c_{\bf e_2}^{\bf i}=d_{\bf e_2}^{\bf i}.$ Substituting this identity into \eqref{eq4.63}, we get 
\begin{equation}\label{eq4.64}
    d_{(m_1, m_2)}^{\bf i}=(m_1+i_1)c_{\bf e_1}^{\bf i}+(m_2+i_2)c_{\bf e_2}^{\bf i}, ~\forall ~(m_1, m_2)\in\mathbb Z^*\times\mathbb Z.
\end{equation}
Thus \begin{equation}\label{eq4.65}
    d_{(m_1, m_2)}^{\bf i}=(m_1+i_1)c_{\bf e_1}^{\bf i}+(m_2+i_2)c_{\bf e_2}^{\bf i}, ~\forall ~(m_1, m_2)\in\mathbb Z^2\setminus\{{\bf 0}\}.
\end{equation}
Taking ${\bf m}=(m_1, 0)$ and ${\bf n}=(0, m_2)$ in identity \eqref{eq4.18}, we obtain 
\begin{equation}\label{eq4.66}
     i_1c_{\bf e_1}^{\bf i}+i_2c_{\bf e_2}^{\bf i}=0。
\end{equation}
Then by \eqref{eq4.64} we obtain $$ d_{(m_1, m_2)}^{\bf i}=m_1c_{\bf e_1}^{\bf i}+m_2c_{\bf e_2}^{\bf i}, ~\forall ~(m_1, m_2)\in\mathbb Z^2\setminus\{{\bf 0}\}.$$
Therefore, for all ${\bf m}\in\mathbb Z^2\setminus\{{\bf 0}\},$ 
$$D_{\bf i}(L_{\bf m})=(m_1c_{\bf e_1}^{\bf i}+m_2c_{\bf e_2}^{\bf i})L_{\bf m+i}=(m_1i_2\frac{c_{\bf e_1}^{\bf i}}{i_2}+m_2(-i_1)\frac{c_{\bf e_1}^{\bf i}}{-i_1})L_{\bf m+i},$$
$$D_{\bf i}(G_{\bf m})=(m_1c_{\bf e_1}^{\bf i}+m_2c_{\bf e_2}^{\bf i})L_{\bf m+i}=(m_1i_2\frac{c_{\bf e_1}^{\bf i}}{i_2}+m_2(-i_1)\frac{c_{\bf e_1}^{\bf i}}{-i_1})G_{\bf m+i}. $$
Putting $\lambda_{\bf i}=\frac{c_{\bf e_1}^{\bf i}}{i_2},$ then by \eqref{eq4.66} we know
 $\frac{c_{\bf e_1}^{\bf i}}{-i_1}=\lambda_{\bf i}, $ and 
 $$D_{\bf i}(L_{\bf m})=\lambda_{\bf i}\det\binom{\bf m}{\bf i}L_{\bf m+i}=\lambda_{\bf i}{\rm ad}L_{\bf i}(L_{\bf m}),$$
 $$D_{\bf i}(G_{\bf m})=\lambda_{\bf i}\det\binom{\bf m}{\bf i}G_{\bf m+i}=\lambda_{\bf i}{\rm ad}L_{\bf i}(G_{\bf m}).$$
Thus $$D_{\bf i}=\lambda_{\bf i}{\rm ad}L_{\bf i}.$$
\end{proof}
According to Lemma \ref{lem4.3} and Lemma \ref{lem4.4}, one can obtain the following theorem.
\begin{theorem}\label{thm4.5}
   $ {\rm Der}(L(V))={\rm ad}(L(V))+\mathbb Cd_1+\mathbb Cd_2, $ where $d_1$ and $d_2$ are defined as follow:
\begin{eqnarray*}
      \forall ~ {\bf m}\in \mathbb Z^2\setminus\{{\bf 0}\}, ~&&d_1(L_{\bf m})=m_1L_{\bf m}, ~ d_2(L_{\bf m})=m_2L_{\bf m},\\
 &&d_1(G_{\bf m})=m_1G_{\bf m}, ~ d_2(G_{\bf m})=m_2G_{\bf m}.
 \end{eqnarray*}
\end{theorem}
\begin{theorem}
    Let $(L(V), [\cdot, \cdot])$ be the Kantor Lie-double of Virasoro-like algebra. Then there are no non-trivial transposed $1$-Poisson superalgebra structrues defined on $(L(V), [\cdot, \cdot])$.
\end{theorem}
\begin{proof}
Suppose $(L(V), \cdot, [\cdot, \cdot])$ is a transposed $1$-Poisson superalgebra. According to Lemma \ref{lem2.1},  for all $x\in L(V)$, $L_x$ is a superderivation of $(L(V), [\cdot, \cdot])$ and $|L_x|=|x|$. By Lemma \ref{lem4.3} and Lemma \ref{lem4.4}, we know that for all ${\bf m, n}\in \mathbb Z^2\setminus\{{\bf 0}\}, $
\begin{eqnarray*}
    L_{\bf m}\cdot L_{\bf n}&=&D_{L_{\bf m}}(L_{\bf n})=(a_{L_{\bf m}}n_1+b_{L_{\bf m}}n_2)L_{\bf n}+\sum\limits_{{\bf i}\in\mathbb Z^2\setminus\{{\bf 0}\}}c_{\bf i}[L_{\bf i}, L_{\bf n}]\\
    &=&(a_{L_{\bf m}}n_1+b_{L_{\bf m}}n_2)L_{\bf n}+\sum\limits_{{\bf i}\in\mathbb Z^2\setminus\{{\bf 0}\}}c_{\bf i}\det\binom{\bf n}{\bf i}L_{\bf n+i}\\
     L_{\bf n}\cdot L_{\bf m}&=&D_{L_{\bf n}}(L_{\bf m})=(a_{L_{\bf n}}m_1+b_{L_{\bf n}}m_2)L_{\bf m}+\sum\limits_{{\bf i}\in\mathbb Z^2\setminus\{{\bf 0}\}}d_{\bf i}[L_{\bf i}, L_{\bf m}]\\
     &=&(a_{L_{\bf n}}m_1+b_{L_{\bf n}}m_2)L_{\bf m}+\sum\limits_{{\bf i}\in\mathbb Z^2\setminus\{{\bf 0}\}}d_{\bf i}\det\binom{\bf m}{\bf i}L_{\bf m+i}\\
          G_{\bf m}\cdot L_{\bf n}&=&D_{G_{\bf m}}(L_{\bf n})=\sum\limits_{{\bf i}\in\mathbb Z^2\setminus\{{\bf 0}\}}e_{\bf i}[G_{\bf i}, L_{\bf n}]=\sum\limits_{{\bf i}\in\mathbb Z^2\setminus\{{\bf 0}\}}e _{\bf i}\det\binom{\bf n}{\bf i}G_{\bf n+i}\\
          L_{\bf n}\cdot G_{\bf m}&=&D_{L_{\bf n}}(G_{\bf m})=(a_{L_{\bf n}}m_1+b_{L_{\bf n}}m_2)G_{\bf m}+\sum\limits_{{\bf i}\in\mathbb Z^2\setminus\{{\bf 0}\}}d_{\bf i}[L_{\bf i}, G_{\bf m}]\\
          &=&(a_{L_{\bf n}}m_1+b_{L_{\bf n}}m_2)G_{\bf m}+\sum\limits_{{\bf i}\in\mathbb Z^2\setminus\{{\bf 0}\}}d_{\bf i}\det\binom{\bf m}{\bf i}G_{\bf m+i}\\
           G_{\bf m}\cdot G_{\bf n}&=&D_{G_{\bf m}}(G_{\bf n})=\sum\limits_{{\bf i}\in\mathbb Z^2\setminus\{{\bf 0}\}}e_{\bf i}[G_{\bf i}, G_{\bf n}]=\sum\limits_{{\bf i}\in\mathbb Z^2\setminus\{{\bf 0}\}}e_{\bf i}L_{\bf n+i}\\
           G_{\bf n}\cdot G_{\bf m}&=&D_{G_{\bf n}}(G_{\bf m})=\sum\limits_{{\bf i}\in\mathbb Z^2\setminus\{{\bf 0}\}}f_{\bf i}[G_{\bf i}, G_{\bf m}]=\sum\limits_{{\bf i}\in\mathbb Z^2\setminus\{{\bf 0}\}}f_{\bf i}L_{\bf m+i}
      \end{eqnarray*}
Since $\cdot$ is supercommutative and $(L_{\bf m}, G_{\bf m})_{{\bf m}\in\mathbb Z^2\setminus\{0\}}$ is linear independent, we deduce 
$\forall ~{\bf i}\in\mathbb Z^2\setminus\{{\bf 0}\}, e_{\bf i}=f_{\bf i}=0, c_{\bf i}\det\binom{\bf n}{\bf i}=d_{\bf i}\det\binom{\bf m}{\bf i}=0, a_{L_{\bf m}}n_1+b_{L_{\bf m}}n_2=a_{L_{\bf n}}m_1+b_{L_{\bf n}}m_2=0.$

Therefore, for all ${\bf m, n}\in \mathbb Z^2\setminus\{{\bf 0}\}, $ 
$$L_{\bf m}\cdot L_{\bf n}=G_{\bf m}\cdot G_{\bf n}=G_{\bf m}\cdot L_{\bf n}=L_{\bf m}\cdot G_{\bf n}=0.$$
 
Thus $(L(V), \cdot, [\cdot, \cdot])$ is a trivial transposed $1$-Poisson superalgebra.
\end{proof}
In summary, the following conclusion is drawn:
\begin{theorem}
  The Lie superalgebra $L(V)$ admits no nontrivial transposed $\delta$-Poisson superalgebra structures.  
\end{theorem}
\bmhead{Acknowledgements}

The authors would like to thank the referee for helpful
comments and suggestions. The work is supported by the NSFC (12201624).

\bmhead{Conflict of interest statement}

The authors have no confilicts to disclose.

\bibliography{sn-bibliography}

\end{document}